\documentclass[11pt]{article}

\usepackage{amssymb,latexsym,verbatim}
\usepackage{graphicx,epsfig,epstopdf,amssymb,color}
\usepackage{amsthm,amsmath}

\newcommand{\eps}{\varepsilon}
\newcommand{\be}{\begin{equation}}
\newcommand{\ee}{\end{equation}}
\newcommand{\bea}{\begin{eqnarray}}
\newcommand{\eea}{\end{eqnarray}}

\newcommand{\ba}{\begin{array}}
\newcommand{\ea}{\end{array}}
\newcommand{\R}{\mathbb{R}}
\newcommand{\Z}{\mathbb{Z}}

\newcommand{\sinc}{\mathrm{sinc}}

\newtheorem{thm}{Theorem}[section]
\newtheorem{lemma}[thm]{Lemma}
\newtheorem{prop}[thm]{Proposition}

\newtheorem*{rmk}{Remark}

\begin{document}
\baselineskip=14pt
\title{Counterpropagating two-soliton solutions in the FPU lattice}
\author{A. Hoffman and C.E.Wayne \\ 
Boston University \\ 
Department of Mathematics and Statistics and Center for BioDynamics \\
111 Cummington St. Boston, MA 02135}

\maketitle
\abstract{
We study the interaction of small amplitude, long wavelength solitary waves in the Fermi-Pasta-Ulam model with general nearest-neighbor interaction potential.  We establish global-in-time existence and stability of counter-propagating solitary wave solutions.  These solutions are close to the linear superposition of two solitary waves for large positive and negative values of time; for intemediate values of time these solutions describe the interaction of two counterpropagating pulses.  These solutions are stable with respect to perturbations in $\ell^2$ and asymptotically stable with respect to perturbations which decay exponentially at spatial $\pm \infty$.}

\section{Introduction}
It has long been of interest in applied math and physics to study how energy and mass are scattered or transfered during the collision of two or more coherent objects.  We study this question for the Fermi-Pasta-Ulam (FPU) problem:
\begin{equation}
\ddot{q}_j = V'(q_{j+1} - q_j) - V'(q_j - q_{j-1}) \qquad j \in \Z
\label{eq:FPU_q}
\end{equation}
which models an infinite chain of anharmonic oscillators with nearest-neighbor interaction potential $V$.  
On making the change of variables $r_j = q_{j+1} - q_j$ and $p_j = \dot{q}_j$, the state variable $u = (r,p)$ satisfies a system of first order Hamiltonian ODEs, 
\begin{equation}
u_t = JH'(u)
\label{FPU}
\end{equation}
where the Hamiltonian $H$ is given by 
\begin{equation}
H(r,p) = \sum_{k \in \Z} \frac{1}{2} p_k^2 + V(r_k)
\label{eq:Hamiltonian}
\end{equation}
The symplectic operator $J$ is given by $J = \left( \ba{cc} 0 & S-1 \\ 1 - S^{-1} & 0 \ea \right)$ where $S$ is the left shift on bi-infinite sequences, i.e. $(Sx)_n := x_{n+1}$.  The problem is well posed in each $\ell^p$ space, but for concreteness and simplicity we work in $\ell^2$.  

The restriction of this problem to a finite lattice with Dirichlet boundary conditions was famously studied in 1955 on the MANIAC computer in order to determine rates of convergence to equipartition of energy \cite{FPU}.  Surprisingly, no such convergence occured.  Instead, the authors observed nearly quasiperiodic motion in which the energy remained mostly in the first few modes.  Later Zabusky and Kruskal \cite{ZabuskyKruskal} rediscovered the equation of Korteweg and DeVries (KdV) as a long-wavelength low amplitude limit of FPU and posited a connection between nearly recurrent states in FPU and soliton interaction in the KdV equation.  The KdV equation has gone on to become both important and well understood as an example of a completely integrable nonlinear dispersive wave equation.  For an excellent review of this productive period see Miura \cite{Miura}.  

Less understood is the theory of solitary waves and their interaction in non-integrable dispersive wave equations.  Martel, Merle, and their collaborators have studied the generalized KdV (gKdV) equation where they have obtained existence and stability of asymptotic multi-soliton states \cite{MMT}, as well as the existence and stability of ``nearly two-soliton solutions,''\cite{MMarXiv} that is, solutions which remain close to the linear superposition of a broad, small, slow-moving solitary wave and a narrow, large, fast-moving solitary wave for all time.  So long as initially the large wave is to the left of the small wave, the large wave will overtake the small wave in finite time.  Thus these solutions describe the collision between two solitary waves.  Here the solitary waves survive the collision despite the fact that the system is not integrable, but unlike the integrable case, the collision is inelastic - energy is transferred or scattered from the solitary waves to dispersive or radiative modes.
   
In this paper, we examine the interaction of solitary waves in the FPU model which unlike gKdV admits two-way wave motion.  In particular, rather than study the analogue of the KdV $2$-soliton solution in this model, we focus on the interaction of counterpropagating waves.  We restrict attention to the ``KdV'' regime of long-wavelength, low-amplitude initial data.  In this near-integrable regime, solitary waves and their stability are well understood.  In particular, see \cite{PW,PF1,PF2,PF3,PF4,Mizumachi} and the references therein.  Our results are consistent with those of Martel and Merle, and the reasonable conjecture that in the near-integrable regime, collisions are nearly elastic.
 
Our main result yields an open set of initial data $U \subset \ell^2$ such that for each $u_0 \in U$, the solution $u$ of $\eqref{FPU}$ with initial data $u(0) = u_0$ satisfies
$$
\lim_{t \to -\infty} \| u(t) - u_{\underline{c}_+}(t) - u_{\underline{c}_-}(t) \| \le Ce^{-\eps^{-\eta_1}} \qquad  
\lim_{t \to \infty} \|u(t) - u_{\overline{c}_+}(t) - u_{\overline{c}_-}(t) \| \le C\eps^{7/2 - \eta_0} $$
for some constant $C$ which doesn't depend on $\eps$ and small numbers $0 < \eta_1 < \eta_0 < 1/2$.  Here $u_{c_+}$ and $u_{c_-}$ denote right- and left- moving solitary wave solutions to $\eqref{FPU}$ and are order $\eps^{3/2}$ in the $\ell^2$ norm.  The wave speeds at $t = \infty$, which we write as $\overline{c}_\pm$ may be different from the wave speeds at $t = -\infty$ which we denote here by $\underline{c}_\pm$.  This is stated more precisely as Theorem $\ref{thm:Collision}$.  The interpretation is that before the collision, the error is exponentially small compared to the main waves and is due only to the weak interaction of the tails of the solitary waves.  However, the collision itself transfers or scatters energy from the coherent structures $u_{c_+}$ and $u_{c_-}$ to radiative modes and perhaps smaller coherent modes.  This scattering manifests in the error term for large positive times which is much larger than the error for large negative times, but remains very small compared with the size of the main waves. 

The strategy for the proof is to consider three different regimes: 
\begin{enumerate}
\item The pre-interaction regime where the solitary waves are well separated and moving towards each other.
\item The interaction regime where the solitary waves are not well separated.
\item The post-interaction regime where the solitary waves have already collided and are now well separated and are moving away from each other.
\end{enumerate}

The interaction regime is handled with a finite time energy estimate which is made precise in Theorem $\ref{thm:Finite_Time_Collision}$.  The pre- and post- interaction regimes are handled with a stability result which is made precise in Theorem $\ref{thm:Orbital_Stability}$.  In order for these theorems to work together to describe the collision of solitary waves, the stability theorem must (a) be valid for the perturbations incurred by the finite time approximation and (b) show that these perturbations remain small compared to the main solitary waves.  However, consider typical stability results, and in particular the ones in \cite{PW,PF4,Mizumachi} which have the following form: 
\begin{quote}
There exists a $\delta > 0$ and $C > 0$ such that whenever the initial perturbation is smaller\footnote{Here, in the interest of space, we are being intentionally vague about what we mean by ``perturbation'' and what we mean by ``small''.  This will be made clear later in the paper.} than $\delta_0 < \delta$ then the perturbation will remain smaller than $C \delta_0$ for all time.  
\end{quote}
In order to extend the stability results of \cite{PW,PF4,Mizumachi} to the two-soliton case, we must quantify both $\delta_0$ and $C$ above.  If $\delta_0$ were to be less than the error incurred by either the weak interaction or the finite-time error estimate, that would present a fundamental mathematical obstruction to extending stability results for a single solitary wave to the two-wave context via the method we present here.  Also, if $C$ were to be larger than the size of the main wave $u_{c_+}$ or $u_{c_-}$ divided by the error incurred by either the weak interaction or the finite time error estimate, then although we could prove a theorem which controls the size of the perturbation, we would be unable to say that the perturbation remains small compared to the main waves.  It turns out that in the stability analysis for solitary waves in FPU, the part of the perturbation which is localized near the solitary waves is relatively innocuous while any perturbation which is not localized is dangerous.  Our main stability result, Theorem $\ref{thm:Orbital_Stability}$, is stated in the form: 
\begin{quote}
There are positive numbers $\delta_{loc} \sim \eps^{3 + \eta}$, $\delta_{nonloc} \sim \eps^{9/2 + \eta}$, $C_{loc} \sim 1$, and $C_{nonloc} \sim \eps^{-3/2}$ such that so long as the (non)localized perturbation is smaller than $\delta_{(non)loc}$ then the perturbation remains smaller than $C_{loc}\delta_{loc} + C_{nonloc}\delta_{nonloc}$ for all time.  
\end{quote}
This allows us to prove that for $\eps$ sufficiently small, there are global-in-time solutions which are close to the linear superposition of two counter-propagating solitary waves for all positive and negative time.  Additionally the methods presented in this paper suggest a blueprint for studying the collision of small broad waves in a variety of dispersive systems which have two way wave motion.

We leave open the important problems of whether or not asymptotic multi-soliton states exist, i.e. whether or not initial data may be chosen so that the radiation terms decay to $0$ as $t \to \infty$, whether or not exact multi-soliton states exist, i.e. whether the initial data may be chosen so that the radiation decays to $0$ as $|t| \to \infty$, as well as a description of the collision of larger solitary waves.

\section{Main Results}

We assume that the potential $V$ satisfies 
\begin{itemize}
\item {\bf H1} \qquad $V \in C^4$ \qquad $V(0) = V'(0) = 0$ \qquad $V''(0) =1$ \qquad $V'''(0) = 1$.  
\end{itemize}

We remark that while requiring that $V''(0)$ and $V'''(0)$ both equal $1$ may appear restrictive, in fact it is not.  To see this, suppose $\tilde{V}(x) = \frac{1}{2}a x^2 + \frac{1}{6} b x^3 + \mathcal{O}(x^4)$ is a general potential with $\tilde{V}(0) = \tilde{V}'(0) = 0$ and $\tilde{V}''(0) > 0$ and $\tilde{V}'''(0) \ne 0$.  Write Newton's equations for a chain of coupled oscillators with interaction potential $\tilde{V}$:
$$\ddot{q}_n = \tilde{V}'(q_n - q_{n-1}) - \tilde{V}'(q_{n+1} - q_n).$$ 
Now make the change of variables $r_n(t) = \alpha( q_n(\beta t) - q_{n-1}(\beta t))$ and $p_n(t) = \alpha \beta \dot{q}_n(\beta t)$.  Then $u = (r,p)$ solves $\eqref{FPU}$ with $V$ given by 
$V(x) = \frac{1}{2} a \beta^2 x + \frac{1}{6} b (\beta^2/\alpha) + \mathcal{O}(x^4)$.
Choosing $\beta = a^{-1/2}$ and $\alpha = b/a $ gives us ${\bf H1}$.

It was proven in \cite{PF1,PF2,PF3,PF4} that under the hypothesis {\bf H1}, there is a $c_{upper} > 1$ and a family of solitary waves $u_c$ parameterized by wave speed $c \in (1,c_{upper}]$ for which $u(k,t) = u_c(k-ct - \gamma)$ is a solution for any $\gamma \in \R$ and moreover the following hold
\begin{itemize}
\item {\bf P1} The profile $r_c$ takes values in an open interval $I_* \subset \R$ containing $0$ such that $V''(r) > 0$ whenever $r \in I_*$.  Moreover $r_c$ is even and decays at the exponential rate $\mathrm{exp}(\kappa(c)|x|)$ where $\kappa(c)$ is the unique positive root of $\mathrm{sinh} \frac{1}{2} \kappa / \frac{1}{2} \kappa = c$.
\item {\bf P2} For each $c_{lower} \in (1,c_{upper})$ and for each $a \in (0,\kappa(c_{lower})/2)$, the map \\ $(\tau,c) \mapsto e^{a \cdot} u_c(\cdot - \tau)$ is $C^1$ from $\R \times (c_{lower},c_{upper})$ into $\ell^2$.  
\item {\bf P3} The energy of the solitary wave $H(u_c)$ satisfies $\frac{dH(u_c)}{dc} = \eps \int_\R \phi_\beta^2 + \mathcal{O}(\eps^3)$ where $\phi_\beta$ is the KdV soliton.  In particular $\frac{dH(u_c)}{dc} \ne 0$ for $\eps$ sufficiently small.  
\end{itemize} 

Here ${\bf P1}$ - ${\bf P3}$ were proven in \cite{PF1}.  Note that due to the time reversability of the equation $\eqref{FPU}$, the properties ${\bf P1 - P3}$ also hold for $c \in [-c_{upper}, -1)$.  In our application of {\bf P2} we require the map $(\tau,c) \mapsto e^{2a}u_c(\cdot - \tau)$ to be $C^3$. We show that it has this additional smoothness in the proof of Lemma $\ref{lem:small_eps_2}$.  

In the remainder of the paper we assume ${\bf H1}$.  Let $c_{upper}$ be as in ${\bf P1}$ - ${\bf P3}$, choose $c_{lower} \in (1,c_{upper})$ and fix $a = a(c_{lower})$ so that ${\bf P1}$ - ${\bf P3}$ hold whenever $\pm c \in (c_{lower},c_{upper})$.  Throughout the paper $K$ and $C$ refer to generic constants.  

Before we may precisely state our main results, we need to develop the weighted function spaces which we will work in.  Let $\ell^2$ be the Hilbert space of two-sided real-valued sequences $\{x_k\}_{k \in \Z}$ which satisfy $\| x \|^2 := \sum_{k \in \Z} x_k^2 < \infty$.  Denote by $\ell^2_a$ the weighted Hilbert space of two-sided real-valued sequences $\{x_k\}_{k \in \Z}$ which satisfy $\| x \|_a^2 := \sum_{k \in \Z} e^{2ak} x_k^2 < \infty$.  Given real valued functions of a real variable $\tau_+(t)$ and $\tau_-(t)$ define $\| x \|_{\tau_+(t)} := e^{-a\tau_+(t)} \| x \|_a = \left(\sum_{k \in \Z} e^{2a(k - \tau_+(t))}x_k^2\right)^\frac{1}{2}$.  Note that the weight $e^{ak}$ in the weighted norm $\| \cdot \|_{\tau_+(t)}$ travels in the frame of a wave $u_{c_+}(\cdot - \tau_+(t))$.  Similarly, we define $\| x \|_{\tau_-(t)} := e^{a\tau_-(t)} \| x \|_{-a}$.  In the following the function $\tau_\alpha$ is allowed to vary depending on the context.  For example, in $\eqref{eq:thm2_3}$ we take $\tau_\alpha \equiv \tau_\alpha^*$  for $\alpha \in \{+,-\}$ and in the statement of Theorem $\ref{thm:Finite_Time_Collision}$ we take $\tau_\alpha = \alpha (c_\alpha t - T_0)$.  Aside from these two theorems, $\tau_\alpha$ will denote the phase of a solitary wave which is allowed to vary in a way fixed by the modulation equations in section $3$.  In the following we will abuse notation, writing $\| x(t) \|_\alpha$ to denote $\| x(t) \|_{\tau_\alpha(t)}$.  We also write $x_+$ and $x_-$ to be the spatially right and left parts of $x$ respectively.  That is $x_+(k) = x(k)$ when $k \ge 0$ and $x_+(k) = 0$ otherwise; similarly $x_-(k) := x(k) - x_+(k)$.

We may now state precisely our main result.
\begin{thm} \label{thm:Collision}
Assume (H1).  For each $0 < \eta_1 < \eta_0 < 1/2$, there is an $\eps_0 > 0$ such that for each $0 < \eps < \eps_0$ there is a constant $C > 0$ and an open set $U \subset \ell^2$ such that for each $u_0 \in U$ there are smooth real valued functions of a real variable $c_+$, $c_-$, $\tau_+$, and $\tau_-$ such that if $u$ denotes the solution of $\eqref{FPU}$ with $u(0) = u_0$, then its perturbation from the sum of two solitary waves $$v(t) := u(\cdot,t) - u_{c_+(t)}(\cdot - \tau_+(t)) - u_{c_-(t)}(\cdot - \tau_-(t))$$ satisfies 
\begin{equation}
\| v(t) \| \le C e^{-\eps^{-\eta_1}} \qquad t < 0  \label{eq:thm1_1}
\end{equation}
and
\begin{equation}
\| v(t) \| \le C \eps^{7/2 - 2\eta_0} \qquad t > 0.
\label{eq:thm1_2}
\end{equation}
Moreover, the weighted perturbation decays:
\begin{equation}
\lim_{t \to \infty} \sum_{k \in \Z} (e^{2a(k - \tau_+(t))} + e^{-2a(k - \tau_-(t))}) v(k,t)^2 = \lim_{t \to -\infty} \sum_{k \in \Z} (e^{-2a(k - \tau_+(t))} + e^{2a(k - \tau_-(t))}) v(k,t)^2 = 0  \label{eq:thm1_3}
\end{equation}

Furthermore, the limits
$$ \lim_{t \to \infty} (c_+(t),c_-(t),\dot{\tau}_+(t),\dot{\tau}_-(t)) \qquad \mbox{ and } \qquad  \lim_{t \to -\infty} (c_+(t),c_-(t),\dot{\tau}_+(t),\dot{\tau}_-(t))$$
exist with $\lim_{t \to \pm \infty} \dot{\tau}_\alpha(t) - c_\alpha(t) = 0$ for $\alpha \in \{+,-\}$.

In addition, there is a subset of $U$ which lies in the weighted space $U_0 \subset \ell^2_a \cap \ell^2_{-a} \cap U$ and which is open in $\ell^2_a \cap \ell^2_{-a}$ such that whenever $u_0 \in U_0$ the weighted perturbations decay at an exponential rate.  That is, there is a $b > 0$ such that the solution $u$ of $\eqref{FPU}$ with $u(0) = u_0$ satisfies
\begin{equation}
\sum_{k \in \Z} (e^{2a(k - \tau_+(t))} + e^{-2a(k - \tau_-(t))}) v(k,t)^2 \le C e^{-bt} \qquad t > 0  \label{eq:thm1_4}
\end{equation}
and
\begin{equation}
\sum_{k \in \Z} (e^{-2a(k - \tau_+(t))} + e^{2a(k - \tau_-(t))})v(k,t)^2 \le C e^{bt} \qquad t < 0  \label{eq:thm1_5}
\end{equation}

\end{thm}

Theorem $\ref{thm:Collision}$ asserts the existence of solutions which remain close to the linear superposition of two counter-propagating solitary waves forever, both forward and backward in time.  In particular, solitary waves persist after collision.  Moreover, the energy which is transfered from the solitary waves to radiative modes or smaller solitary waves is small relative to the main waves, whose $\ell^2$ norms are $\mathcal{O}(\eps^{3/2})$.  Furthermore, these counter-propagating wave solutions are stable with respect to perturbations in $\ell^2$.  In addition, solutions which differ from the sum of two counter-propagating solitary waves by a sufficiently small exponentially localized perturbation at some initial time, continue to differ from the sum of solitary waves by a small exponentially localized amount, both forward and backward in time.

In our proof we combine an orbital stability result for weakly interacting waves with a finite time approximation result for strongly interacting waves.  There is a tradeoff here between the exponential control of the radiation at $t < 0$ and the polynomial control over the radiation for $t > 0$.  This tradeoff is a result of choosing the time period over which one uses the finite time approximation.  If one uses a relatively short time period, so that the waves just have time to pass through each other, then the polynomial error incurred by the approximation is smaller.  This corresponds to a small value of $\eta_0$ above.  However, this has consequences for the bounds in the orbital stability result because if the waves start out relatively close to each other, then the cross terms which come from the interaction will be larger.  This corresponds to a small value of $\eta_1$ above.  On the other hand, if one uses finite time approximations for a larger time period, then the error incurred by the approximation will be larger, but the interaction terms before and after this period will be smaller, corresponding to relatively large values of $\eta_1$ and $\eta_0$.
 
The orbital stability and finite time approximation results are of independent interest and we now state them as separate theorems.  

\begin{thm} \label{thm:Orbital_Stability}
Assume (H1).  For each $\eta_1 > 0$ there are positive constants $C$, $\delta_{loc}$, $\delta_{nonloc}$, and $\eps_0$ such that for any $\eps \in (0,\eps_0)$ the following hold.  Suppose that there is a decomposition
$$ u_0(\cdot) - u_{c_+^*}(\cdot - \tau_+^*) - u_{c_-^*}(\cdot - \tau_-^*) = v_{10} + v_{20}$$ 
such that $\| v_{10} \| < \delta_{nonloc}$ and $\| (v_{20}) \|_+ + \| (v_{20})_-\|_- + \| v_{20} \| < \delta_{loc}$
with $\tau_-^* < - C \eps^{-1 -\eta_1}  < C \eps^{-1 -\eta} < \tau_+^*$, $\delta_{nonloc} < C \eps^{9/2 + \eta}$, and $\delta_{loc} < C \eps^{3 + \eta}$.
Then there are smooth real valued functions of a positive real variable $c_+$, $c_-$, $\tau_+$, and $\tau_-$ such that the solution $u$ of $\eqref{FPU}$ with $u(0) = u_0$ satisfies
\begin{equation}
 \|u(t,\cdot) - u_{c_+(t)}(\cdot - \tau_+(t)) - u_{c_-(t)}(\cdot - \tau_-(t)) \| <  C\eps^{-3/2} \delta_{nonloc} + C \delta_{loc} + C e^{-\eps^{-\eta_1^{-1}}} \qquad  \mbox{ for all } t > 0
 \label{eq:thm2_1}
 \end{equation}
and
\begin{equation}
\lim_{t \to \infty} \|u(t,\cdot) - u_{c_+(t)}(\cdot - \tau_+(t)) - u_{c_-(t)}(\cdot - \tau_-(t)) \|_\alpha = 0
\label{eq:thm2_2}
 \end{equation} 
Moreover, the limit
$ \lim_{t \to \infty} (c_+(t),c_-(t),\dot{\tau}_+(t),\dot{\tau}_-(t))$
exists with $\lim_{t \to \infty} \dot{\tau}_\alpha(t) = c_\alpha(t)$ for $\alpha \in \{+,-\}$.

Furthermore, if $v_{10} = 0$, then 
\begin{equation}
\| v(t)_\alpha \|_\alpha \le Ce^{-bt} \qquad \alpha \in \{+,-\}
\label{eq:thm2_3}
\end{equation}
where in the small $\eps$ regime $b \sim \eps^3$.
\end{thm}

Note that we are fairly quantitative in determining the small $\eps$ asymptotics for the initial data whose evolution we are able to control as well as for the degree to which we are able to control it.  This is an essential feature of our method as the stability result must be able to tolerate the error incurred by the following finite time estimate.

\begin{thm} \label{thm:Finite_Time_Collision} 
Assume $H1$.  For any $0 < \eta_1 < \eta_0 < 1$, there exist positive constants  $\eps_0$, $C_0$, $C_1$, and $E_0$ such that if $c_{-} = -1 - \frac{1}{12} \eps^2 \beta_{-}$ and $c_{+} = 1+ \frac{1}{12} \eps^2 
\beta_{+}$, with $|\beta_{\pm}| \le 1$ and $0 < \eps < \eps_0$ and $0 < E < E_0$, then whenever $u$ is a solution of $\eqref{FPU}$ with 
$$ \| u(0) - u_{c_+}(\cdot - \tau_+) - u_{c_-}(\cdot - \tau_-) \| <  \eps^{11/2-\eta_0} E $$
for some $\eta > 0$, then the difference between the solution of $\eqref{FPU}$ with initial condition $u(0)$ and the linear superposition of waves has a decomposition
$$ u(t) -  u_{c_+}(\cdot - c_+ t - \tau_+) - u_{c_-}(\cdot - c_- t - \tau_-) = v_1(t) + v_2(t)$$ 
such that 
$$\| v_1(t) \| \le C_0\eps^{11/2 - 3 \eta_0}(E + 1)$$
and
$$ \|v_2(t)\| + \| v_2(t) \|_+ + \| v_2(t)\|_- \le C_0 \eps^{7/2 - \eta_0}(E + 1) $$
hold for $ 0 < t < C_0 \eps^{-1 - \eta_1}$.
If in addition, $$\| u(0) - u_{c_+}(\cdot - \tau_+) - u_{c_-}(\cdot - \tau_-)\|_+ + \| u(0) - u_{c_+}(\cdot - \tau_+) - u_{c_-}(\cdot - \tau_-)\|_-  \le
C_1 \eps^{7/2 - \eta_0} $$ then we may take $v_1 \equiv 0$.
\end{thm}

The proof of Theorem $3$ follows from a straightforward, albeit nonstandard, energy estimate.  The following lemma will help to streamline the proof.  

\begin{lemma} \label{lem:dWE}
Consider the nonlinear inhomogeneous discrete wave equation
$$ \left\{ \ba{l} \dot{r} = (S - I)p \\ \\ \dot{p} = (I - S^{-1})r + g(t,r,p) \ea \right.$$
Let $T_0 \in \R$ and $a \sim \eps$ be given and let $c_\alpha$ be given such that $\alpha c_\alpha - 1 \sim \eps^2$ for $\alpha \in \{+,-\}$.  Let $\| \cdot \|$ denote the standard $\ell^2$ norm and let $\| \cdot \|_\alpha$ denote the weighted norm with phase $\tau_\alpha = c_\alpha t - \alpha T_0$.

Assume that there are positive constants $C_0$, $C_1$, and $\eta_0 \in (0,1)$ which do not depend on $\eps$ such that for $\eps$ sufficiently small we have $\| g(t,r,p) \| \le C_0\eps^{1 + \eta_0}$ for $0 < t < C_1\eps^{-1 - \eta_0}$
Then  
$$ \| r(t) \|^2 + \|p(t) \|^2 \le e^{C_0C_1}(\|r(0)\|^2 + \|p(0)\|^2 + 1) \qquad 0 \le t \le  C_1 \eps^{-1 - \eta_0}$$
Similarly, if $\| g(t,r,p) \|_\alpha \le C_0\eps^{1 + \eta_0}$ for $0 < t < C_1 \eps^{-1 - \eta_0}$ then for $\eps$ sufficiently small 
$$\| r(t) \|_\alpha^2 + \|p(t)\|_\alpha^2 \le e^{2C_0C_1} (\|r(0)\|^2 + \|p(0)\|^2 + 1) \qquad 0 \le t \le C_1 \eps^{-1 - \eta_0}$$
\end{lemma}

\begin{proof}
We consider first the unweighted $\ell^2$ norm.  Let $\mathcal{E}(t) := \| r(t) \|^2 + \|p(t)\|^2$.  Compute 
$$
 \frac{d}{dt} \mathcal{E}(t)  =  \langle r, (S - I)p \rangle + \langle p, (I - S^{-1})r \rangle + \langle p, g \rangle = \langle p, g \rangle \le \| g \| ( \mathcal{E} + 1).
$$ 
In the second equality we have used the fact that $S^* = S^{-1}$ when regarded as an operator on $\ell^2$.  In the last inequality we have used the Cauchy-Schwartz inequality as well as the fact that $x < x^2 +1$ to replace $\|p\|$ with $\mathcal{E} + 1$.  From Gronwall's inequality and the smallness condition on $g$ we obtain
$$ \mathcal{E}(t) \le e^{C_0C_1}\mathcal{E}(0) + e^{C_0C_1} - 1 \qquad 0 < t < C_1\eps^{-1 - \eta_0} $$ as desired.

Now consider the weighted norm.  Let $\mathcal{E}_\alpha(t) := \| r(t) \|_\alpha^2 + \| p(t) \|_\alpha^2$ and let $\langle x, y \rangle_\alpha := \sum_{k \in \Z} e^{\alpha a (k - c_\alpha t)} x_k y_k$ so that $\| x \|_\alpha^2 = \langle x, x \rangle_\alpha$.  Compute
$$ \ba{lll} 
\frac{d}{dt} \mathcal{E}_\alpha(t) & = & - 2\alpha ac_\alpha \mathcal{E}_\alpha + \langle (S^* - S^{-1})r, p \rangle_\alpha + \langle p, g \rangle_\alpha \\ \\
& \le & \left( e^a(1 - e^{-2a}) - 2a\alpha c_\alpha \right)\mathcal{E}_a + \| g \|_\alpha( \mathcal{E}_\alpha + 1) \\ \\
& \le & 2C_0 \eps^{1 + \eta_0} (\mathcal{E}_\alpha + 1) \qquad \mbox{ for } \qquad 0 < t < C_1\eps^{-1 - \eta_0}.
\ea
$$
In the second line we have used the fact that $S^* =e^{-2a}S^{-1}$ and $\|S^{-1}\| = e^a$ in the space $\ell^2_a$ to bound the first term and we have used the Cauchy-Schwartz inequality together with the fact that $x \le x^2 + 1$ to bound the second term.  In the third line we have used the fact that $\alpha c_\alpha - 1 \sim \eps^2$ and $a \sim \eps$ to show that the first term is higher order in $\eps$ than the second term and thus is no bigger than the second term for sufficiently small $\eps$.  We have also used the smallness assumption on $g$ to simplify this term.  The result again follows from Gronwall's inequality.
\end{proof}

We can now prove Theorem $\ref{thm:Finite_Time_Collision}$.

\begin{proof}
The idea of the proof is to write 
$u = u_{c_+} + u_{c_-} + v_2 + v_1$ where $v_2$ is localized and serves to cancel the leading order terms in $\frac{d}{dt} \| v_1 \|^2 $.  As a preliminary step we study the following system of equations:
$$ \left\{ \ba{l} \dot{\phi}_1 = (S - I)\phi_2 \\ \\ \dot{\phi}_2 = (I - S^{-1})\left[ \phi_1 + 2\eps^{-7/2 + \eta_0} r_{c_+}r_{c_-} \right] \ea \right.
$$
where $r_{c_+}$ and $r_{c_-}$ are solitary wave solutions of the FPU equation evaluated at $k - c_+ t - T_0$ and $k + c_+ t + T_0$ respectively.
Our aim is to apply Lemma $\ref{lem:dWE}$ to show that $\phi$ remains small on timescales of order $\eps^{-1 - \eta_0}$.
To that end, we estimate 

$$ \ba{lll}
\eps^{\eta_0 - 7/2} \| (I - S^{-1})r_{c_+} r_{c_-} \|_\alpha 
& = &
\eps^{\eta_0 - 7/2} \| r_{c_+}(I-S^{-1})r_{c_-} + (S^{-1}r_{c_-})(I - S^{-1})r_{c_+}  \|_\alpha \\ \\
& \le &  
C\eps^{\eta_0 -1/2} \| r_{c_\alpha}\|_\alpha  + C\eps^{\eta_0 - 3/2} \|(I - S^{-1})r_{c_\alpha}\|_\alpha
\le 
C \eps^{1 + \eta_0}. \ea
$$

In the first inequality in the second line we have used that $\| (I - S^{-1})r_{c_{-\alpha}} \|_{\ell^\infty} \le C\eps^3$ and that $\| r_{c_{-\alpha}} \|_{\ell^\infty} \le C\eps^2$, which follow from $\eqref{eq:KdV_approx}$ and the Mean Value Theorem.  In the second inequality we have used that $\| r_{c_\alpha} \|_\alpha \le C\eps^{3/2}$ and that $\| (I - S^{-1}) r_{c_\alpha} \|_\alpha \le C\eps^{5/2}$ which also follow from $\eqref{eq:KdV_approx}$ and are again established in Lemma $\ref{lem:small_eps_2}$ .  We remark that (at least for $t < T_0$), $\| r_{c_{-\alpha}} \|_\alpha$ is exponentially large, thus it is essential that we are able to use the $\ell^\infty$ norm rather than the $\| \cdot \|_\alpha$ norm to estimate $r_{c_{-\alpha}}$.  The bound for the $\ell^2$ norm is similar.

We consider now the system
$$ \left\{ \ba{l} \dot{\psi}_1 = (S - I) \psi_2 \\ \\ \dot{\psi}_2 = (I - S^{-1})(\psi_1 + 2\eps^{-1 + \eta_0}((r_{c_+} + r_{c_-})\phi_1) \ea \right.$$
We again use Lemma $\ref{lem:dWE}$ to bound both the $\ell^2$ and weighted norms.  Here it suffices to use the fact that $\| \phi_1 \| + \|\phi_1\|_+ + \|\phi_1\|_- < C$ and $\|r_{c_+} + r_{c_-}\|_{\ell^\infty} \le C\eps^2$ to bound the inhomogeneous term.

Now let $v_2 := \eps^{7/2 - \eta_0} \phi + \eps^{9/2 - 2\eta_0} \psi$.   
We define the residual $v_1 = (R,P)$ by the ansatz
$$\ba{l} 
r = r_{c_+} + r_{c_-} + \eps^{7/2 - \eta_0} \phi_1 + \eps^{9/2 - 2 \eta_0}\psi_1 + \eps^{11/2 - 3 \eta_0} R \\ \\ 
p = p_{c_+} + p_{c_-} + \eps^{7/2 - \eta_0} \phi_2 + \eps^{9/2 - 2 \eta_0}\psi_2 + \eps^{11/2 - 3 \eta_0} P \ea $$ where $\phi$ and $\psi$ are as above.

Upon substituting the above ansatz into $\eqref{FPU}$ we see that $R$ and $P$ satisfy the differential equation
$$\left\{ \ba{lll} 
\dot{R} & = & (S - I)P \\ \\ 
\dot{P} & = & (I - S^{-1})R + g \ea \right.$$
where the inhomogeneous term is 
$$\ba{lllr}
g & = & (I - S^{-1}\left[ (r_{c_+} + r_{c_-})(\eps^{\eta_0 - 1} \psi_1 + R)\right] & (i) \\ \\
& & + \eps^{3/2 + \eta_0}(I - S^{-1}) \left[ (\phi_1 + \eps^{1 - \eta_0}\psi_1 + \eps^{2 - 2\eta_0} R)^2\right]  & (ii) \\ \\
& & + 3\eps^{-11/2 + 3\eta_0} (r_{c_+} + r_{c_-})\left[ (I - S^{-1})(r_{c_+}r_{c_-}) \right] + S^{-1}(r_{c_+}r_{c_-})(I - S^{-1})(r_{c_+} + r_{c_-}) \rangle & (iii) \\ \\
& & +\eps^{2\eta_0 - 2}(I - S^{-1})\left[ (\phi_1 + \eps^{1-\eta_0} \psi_1 + \eps^{2-2\eta_0} R) F_3(r_{c_+},r_{c_-},\eps^{7/2 - \eta_0}\phi_1, \eps^{9/2 - 2\eta_0}\psi_1,\eps^{11/2 - 3\eta_0}R)\right] & (iv) \\ \\ 
& & + \eps^{3\eta_0 - 11/2} (I - S^{-1})F_4(r_{c_+},r_{c_-},\eps^{7/2 - \eta_0}\phi_1, \eps^{9/2 - 2\eta_0} \psi_1, \eps^{11/2 - 3\eta_0} R) & (v). 
 \ea $$

Here $F_3$ is a homogeneous polynomial of degree $2$ in its arguments and accounts for the third order terms in the Taylor expansion of $V'$ about $0$ other than those which arise from $(r_{c_+} + r_{c_-})^3$.  The function $F_4$ accounts for the fourth and higher order terms in the Taylor expansion of $V'$ about $0$ and is quartic in all of its arguments.

From fact that $\|r_{c_\alpha}\|_{\ell^\infty} \le C\eps^2$ we see that $(i) \le C \eps^{1 + \eta_0}$ so long as $\|\psi_1\|$ remains $\mathcal{O}(1)$ and $\|R \|$ remains $\mathcal{O}(\eps^{\eta_0 - 1})$.  Similarly $(ii) \le C\eps^{3/2 + \eta_0}$ so long as $\|\phi\|$ remains $\mathcal{O}(1)$, $\|\psi\|$ remains $\mathcal{O}(\eps^{\eta_0 - 1})$, and $\|R\|$ remains $\mathcal{O}(\eps^{2\eta_0 - 2})$.  To bound $(iii)$ we use the fact that $\| r_{c_\alpha} \|_{\ell^\infty} \le C\eps^2$, that $\|(I - S^{-1})(r_{c_+}r_{c_-})\|_{\ell^\infty} \le C\eps^{9/2}$, that $\| S^{-1} (r_{c_+} r_{c_-}) \| \le \|r_{c_+}\|_{\ell^\infty} \|r_{c_-}\| \le C\eps^{7/2}$ and that $\|(I - S^{-1})r_{c_\alpha}\| \le C\eps^3$ to conclude $(iii) \le C \eps^{1 + 3\eta_0}$.

Since $F_3$ is quadratic, we have $(iv) \le C\eps^{2 + 2\eta_0}$ so long as its arguments remain $\mathcal{O}(\eps^2)$ and the coefficient remains $\mathcal{O}(1)$.  That is, we have $(iv) \le C\eps^{2 + 2 \eta_0}$ so long as $\phi$ remains $\mathcal{O}(1)$, $\psi$ remains $\mathcal{O}(\eps^{\eta_0 - 1})$ and $R$ remains $\mathcal{O}(\eps^{2 - 2\eta_0})$.  We similarly conclude that $(v) \le \eps^{2 + 3\eta_0}$ so long as $\|\phi_1\|$ remains $\mathcal{O}(\eps^{\eta_0 - 3/2})$, $\|\psi_1\|$ remains $\mathcal{O}(\eps^{\eta_0 - 5/2})$, and $\|R\|$ remains $\mathcal{O}(\eps^{\eta_0 - 7/2})$.  Note that we have only had to assume that $\| R \|$ remains $\mathcal{O}(\eps^{\eta_0 - 1})$ on timescales $t \sim \eps^{-1 - \eta_0}$.  Thus from Lemma $\ref{lem:dWE}$ it follows that $\|R\|^2 + \| P \|^2$ actually remains $\mathcal{O}(1)$ on timescales $t \sim \eps^{-1 - \eta_0}$.  

To conclude the proof it remains only to show that if the initial perturbation is localized, it remains so for the timescales of interest.
To that end, redefine the residuals $R$ and $P$ by the ansatz
$$r = r_{c_+} + r_{c_-} + \eps^{7/2 - \eta_0}R \qquad p = p_{c_+} + p_{c_-} + \eps^{7/2 -\eta_0}P $$
and substitute into $\eqref{FPU}$ to obtain
\[ \left\{ \ba{l} \dot{R} = (S - I)P \\ \\ \dot{P} = (I - S^{-1})R + g \ea \right.\]
where $g := \eps^{\eta_0 - 7/2} (I - S^{-1}) \left( V'(r_{c_+} + r_{c_-} + \eps^{7/2 - \eta_0}) - V'(r_{c_+}) - V'(r_{c_-})- \eps^{7/2 - \eta_0} R\right) $.
Use Taylor's Theorem to write
$V'(r_{c_+} + r_{c_-}) = V'(r_{c_-}) + V''(r_{c_-})r_{c_+} + V'''(r_{c_-})r_{c_+}^2 + V''''(\theta_1)r_{c_+}^3$ and $V'(r_{c_+} + r_{c_-} + \eps^{7/2 - \eta_0}R) = V'(r_{c_+} + r_{c_-}) + V''(r_{c_+} + r_{c_-})\eps^{7/2 - \eta_0} R + V'''(\theta_2) \eps^{7 - 2\eta_0} R^2$.
Thus
$$
\ba{lll}  g & = & \eps^{\eta_0 - 7/2} (I - S^{-1})\left( 
 (V''(r_{c_-}) - 1)r_{c_+} +
 V'''(r_{c_-})r_{c_+}^2 + 
 V''''(\theta_1)r_{c_+}^3 \right. \\ \\ 
 & &  \qquad + \left.
( V''(r_{c_+} + r_{c_-}) - 1) \eps^{7/2 -\eta_0}R + V'''(\theta_2)\eps^{7 - 2\eta_0} R^2 \right) 
\ea  $$
The quantities $\theta_i$ are bounded uniformly in both space and time, so long as $\|R\|_{\ell^\infty}$ remains bounded.  Since we are trying to prove that $R$ remains bounded, and we must assume that $R$ remains bounded in order to bound $V'''(\theta_i)$, we proceed with some care.  More precisely, assume that there are some positive constants $C_4$ and $C_1$ so that $\|R\|_{\ell^\infty} \le C_4$ for $0 < t < C_1 \eps^{-1 - \eta_0}$.  It follows that there is some positive constant $C_5$ so that $|V''''(\theta_1)| + |V'''(\theta_2)| < C_5$.
Thus
$$ \ba{lll}
\| g \|_+ & \le & \eps^{\eta_0 - 7/2}\left( \|r_{c_-} \|_{\ell^\infty} \|(I - S^{-1}) r_{c_+}\|_+
+ C\|r_{c_+}\|_{\ell^\infty} \|(I - S^{-1})r_{c_+}\|_+ + C_5\|r_{c_+}\|_{\ell^\infty}^2 \|r_{c_+}\|_+ \right) \\ \\
& & \qquad + \; C\|r_{c_+} + r_{c_-}\|_{\ell^\infty} \|R\|_+ + C_5\eps^{7/2 - \eta} \|R\|_{\ell^\infty} \|R\|_+ \\ \\
& \le & C \eps^{1+ \eta_0} + C_5 \eps^{2 - \eta_0} + C\eps^2\|R\|_+ + C_5\eps^{7/2 - \eta_0} \|R\|_{\ell^\infty} \|R\|_+.
\ea
$$
so long as $\|R\|_{\ell^\infty}$ remains bounded.  We can always choose $\eps$ small enough so that 
\begin{equation}
\| g \|_+ \le 2C\eps^{1 + \eta_0}
\label{eq:111}
\end{equation}  where $C$ is a constant which depends upon neither $\eps$ nor $R$.  Thus, so long as so long as the above is valid, i.e. $\|R(t)\|_{\ell^\infty} < C_4$ then from Lemma 4 we have $\|R(t)\|^2_+ + \|P(t)\|_+^2 \le e^{2C C_1}(\|R(0)\|^2 + \|P(0)\|^2)$ for $0 < t < C_1 \eps^{-1 - \eta_1}$.  We may now, {\it a posteriori} choose $C_4 = 2e^{2C C_1}$ and make $\eps$ smaller if necessary so that $\eqref{eq:111}$ holds.

One may obtain similar estimates for $\|g\|_-$ and $\|g\|$.  In particular
$$ \|g\|_+ + \|g\|_- + \|g\| \le C\eps^{1 + \eta_0}$$ so long as $\|R\|_+ + \|R\|_- + \|R\|$ remains $\mathcal{O}(1)$.   
This concludes the proof.
\end{proof}

We now show that Theorem $\ref{thm:Collision}$ follows from Theorems 2 and 3.

\begin{proof}[Proof of Theorem $\ref{thm:Collision}$]
The strategy for the proof is to consider three different regimes: 
\begin{enumerate}
\item The pre-interaction regime where the solitary waves are well separated and moving towards each other.
\item The interaction regime where the solitary waves are not well separated.
\item The post-interaction regime where the solitary waves have already collided and are now well separated and are moving away from each other.
\end{enumerate}

We choose initial data in the pre-interaction regime which is exponentially close (in $\ell^2$) to the sum of two well separated solitary waves.  In light of the symmetry $u(n,t) \mapsto u(-n,-t)$ for $\eqref{FPU}$ we may apply Theorem $\eqref{thm:Orbital_Stability}$ backwards in time. In particular, using $\eqref{eq:thm2_1}$ with $\delta_{loc} = \delta_{nonloc} = e^{-\eps^{-\eta_1}}$ and $\tau_+^* = -\tau_-^*= \eps^{-1-\eta_1}$ yields $\eqref{eq:thm1_1}$.  To control the evolution in the interaction regime we use Theorem $\ref{thm:Finite_Time_Collision}$.  However, note that the definition of the weighted norms change under the symmetry $u(n,t) \mapsto u(-n,-t)$; when the solitary waves are moving away from each other the weighted norms don't see mass that lies in between the main waves, whereas when the solitary waves are moving toward each other they do.  Thus to ensure that the localized error $v_{20}$ is small, we must take $\delta_{nonloc} = e^{-\eps^{-1 - \eta_0}}$ above.  Having done this, the localized error $v_{20}$ is order $\eps^{7/2-\eta_0}$ and the $\ell^2$ error $v_{10}$ is order $\eps^{11/2 - 3\eta_0}$ in the interaction regime.  In light of the symmetry $u(n,t) \mapsto u(n,t-T)$ we may again apply Theorem $\ref{thm:Orbital_Stability}$.  Using $\eqref{eq:thm2_1}$ with $\delta_{loc} = C\eps^{7/2-\eta_0}$ and $\delta_{nonloc} = C\eps^{11/2 - 3\eta_0}$ yields $\eqref{eq:thm1_2}$.  
The weighted norm estimate $\eqref{eq:thm1_3}$ follows from $\eqref{eq:thm2_2}$ and convergence of modulation parameters similarly follows from Theorem $\ref{thm:Orbital_Stability}$.

To complete the proof it remains only repeat this procedure for initial data which is exponentially close to the sum of solitary waves in the weighted space $\ell^2_a \cap \ell^2_{-a}$.  Using $\eqref{eq:thm2_3}$ with $\delta_{nonloc} = 0$, $\delta_{loc} < C\eps^{7/2 +\eta}$ and $\tau_+^* = -\tau_-^*= \eps^{-1-\eta_1}$ yields $\eqref{eq:thm1_4}$.  Upon applying Theorem $\ref{thm:Finite_Time_Collision}$ and using the fact that the initial perturbation is localized, we have $v_{10} = 0$ and $v_{20}$ is order $\eps^{7/2 - \eta_0}$.  
Using $\eqref{eq:thm2_3}$ with $\delta_{loc} = C\eps^{7/2-\eta_0}$ and $\delta_{nonloc} = 0$ yields $\eqref{eq:thm1_5}$. 
\end{proof}

The rest of the paper is devoted to proving Theorem $\ref{thm:Orbital_Stability}$.  Before sketching its proof, we briefly review the previous stability results \cite{PW,PF4,Mizumachi} on which we build.  The KdV equation is
 \begin{equation}
u_t = u_{xxx} + 3(u^2)_x. \label{eq:KdV}
\end{equation}
Traveling waves $u(x,t) = Q_c(x - ct)$ satisfy the wave profile equation $cQ + Q'' + 3Q^2$ which can be solved explicity to obtain $Q_c(\xi) = c \; \mathrm{sech}^2 (\sqrt{c} \xi/2)$.  These solitary waves $Q_c(\cdot - \tau)$ form a $2$-manifold in the space of initial data, parameterized by allowing the speed (or mass) $c$ and the phase $\tau$ to vary independently.  To study perturbations of solitary wave solutions, one makes the ansatz $u(x,t) = Q_{c(t)}(x - \tau(t)) + w(x,t)$.  The evolution of the perturbation $w$ is governed by a dispersive PDE.  In order to control its evolution, Pego and Weinstein \cite{PW} work in an exponentially weighted space and prove that $w(t)$ tends to zero at $t \to \infty$ in this norm.  The evolution of the speed $c(t)$ and phase $\tau(t)$ are governed by ODE's which are coupled to the PDE for $w(t)$.  The control of $w$ in the weighted norm, together with energy estimates which come from the conservation of the Hamiltonian and an orthogonality condition on $w$, is sufficient to imply convergence of $\dot{\tau}(t)$ and $c(t)$ as $t \to \infty$.

In 1999 Friesecke and Pego showed that in the ``KdV'' regime of long-wavelength, small amplitude traveling waves whose speed is close to the ``sonic'' speed, which is $c_{sonic} = \pm1$ in our case, FPU solitary waves could be well-approximated by KdV solitary waves.  More precisely, they showed that if the speed $c$ is sufficiently close to $1$ (or $-1$), and if one writes $c^2 = 1+ \frac{1}{24} \eps^2 $, then the traveling wave profile satisfies the estimate
\begin{equation}\label{eq:KdV_approx}
\| \frac{1}{\eps^2} r_{c}(\frac{\cdot}{\eps}) - \phi_1(\cdot) \|_{H^1(\R) } \le C \eps^2
\end{equation}
where $\phi_1(\xi) = \frac{1}{4}{ \mathrm{sech}}^2(\xi/2)$.  This closeness between KdV and FPU solitary waves is instrumental in importing estimates for KdV linearized about a solitary wave to the FPU setting.  
 
In a series of papers from 2002 to 2004 \cite{PF2, PF3, PF4}, Pego and Friesecke showed that these waves are stable with respect to perturbations which are small in both $\ell^2$ and a weighted space.  Their method of proof is adapted from \cite{PW} and indeed relies heavily on results presented there.  This series of papers provides a blueprint for importing estimates which are known for the KdV equation to equations which have KdV as a long-wave low-amplitude scaling limit.  However this work, as well as that of \cite{PW}, is limited to perturbations which lie in a weighted space and thus says nothing about the behavior of solutions which differ from a solitary wave by a perturbation which decays very slowly.

Orbital stability in the energy space for gKdV solitons was obtained by Martel and Merle in \cite{MM1}.  The conceptually simplest version of the argument relies on a virial identity for the perturbation equation.  That is, small perturbations to the solitary wave never linger near the solitary wave - regardless of whether they are exponentially localized or not.  Since evolution of the modulation parameters $c$ and $\tau$ is controlled by the mass of the perturbation near the solitary wave, one may use this estimate to control the evolution of $c$ and $\tau$, and thus use energy estimates to obtain Lyapunov stability for the solitary wave.

More recently, Mizumachi \cite{Mizumachi} has obtained orbital stability in $\ell^2$ for solitary waves of $\eqref{FPU}$ in the KdV regime via an adaptation of \cite{PF2} which uses ideas present in \cite{MM1}.  Unfortunately we know of no virial identity for the evolution of perturbations to solitary waves in FPU.  However, it is not hard to show that exact solutions of FPU with small initial data move more slowly than larger solitary waves.  Mizumachi decomposes the perturbation into two pieces, one of which is an exact solution and the other of which is exponentially localized \cite{Mizumachi}.  With this decomposition one can show that FPU solitary waves are stable to all sufficiently small perturbations in $\ell^2$ even if they are not exponentially localized.

With this background in mind, we are now ready to outline our proof.
 We first make the ansatz 
\begin{equation} u(t) = u_{c_+(t)}(\cdot - \tau_+(t)) + u_{c_-(t)}(\cdot - \tau_-(t)) + v(t).
\label{eq:ansatz1}
\end{equation}
Here we allow the modulation parameters $c_\alpha$ and $\tau_\alpha$ for $\alpha \in \{+,-\}$ to vary in order to control the neutral modes associated with variation in wavespeed and phase.  This control is achieved via orthogonality conditions which we place on the perturbation $v$.

It is a consequence of the convexity (and conservation) of the Hamiltonian that so long as these orthogonality conditions are satisfied, we obtain 
\begin{equation} \| v(t) \|^2 \le C\left( \|v(t_0)\|^2 + \eps |c_+(t) - c_+(t_0)| + \eps |c_-(t) - c_-(t_0)| \right) + {\tt exp}. \label{eq:Intro_energy}
\end{equation}
Here ${\tt exp}$ denotes the exponentially small terms generated by the interaction between the tails of $u_{c_+}$ and $u_{c_-}$.  Thus to obtain orbital stability it suffices to restrict attention to the region where the interaction terms are small and to control the evolution of the modulation parameters $c_\alpha$.   

Substituting the ansatz $\eqref{eq:ansatz1}$ into $\eqref{FPU}$ we obtain a lattice differential equation for the perturbation $v$ which is coupled to a system of ordinary differential equations for the modulation parameters $c_\alpha$ and $\tau_\alpha$.  If one localizes the perturbation to the left- and right- half lattices by writing $v = v_+ + v_-$ then the linear part of the lattice differential equation for $v_\alpha$ generates a semigroup which was studied in \cite{PF4} and is known to decay exponentially for initial data which are exponentially localized and orthogonal to the neutral modes.  Given arbitrary initial data $v \in \ell^2$, we follow   \cite{Mizumachi} in decomposing it as $v = v_1 + v_2$ where $v_1$ is an exact solution of $\eqref{FPU}$ and $v_2$ is exponentially localized.  We may now further localize $v_2$ by writing $v_2 = v_{2,+} + v_{2,-}$ with $v_{2,\alpha}$ near $u_{c_\alpha}$.  Now $v_{2,\alpha}$ lives in the space on which the semigroup generated by the linear part of its evolution equation decays.  Thus we may write a variation of constants formula for $v_{2,\alpha}$.  However, this equation is now coupled to the exact solution part of the perturbation $v_1$ as well as the modulation parameters $c_\alpha$ and $\tau_\alpha$ and cross terms associated to $u_{c_{-\alpha}}$.  To give a sense for the relationship between the quantities involved, we write the variation of constants formula here: 

$$ \| v_{2,\alpha}(t) \|_\alpha  \le  Ce^{-b(t-t_0)}(\|v_{2,\alpha}(t_0)\|_\alpha + \|v_{1,\alpha}(t_0)\|_{W_\alpha}) +  C\| v_{1,\alpha}(t) \|_{W_\alpha} + C \int_{t_0}^t e^{-b(t-s)} \|G_\alpha (s)\|_\alpha ds$$ 
where
\[ \ba{lll} \|G_\alpha(s)\| & \le & 
\eps^2 \| v_{1,\alpha} (s) \|_{W_\alpha} +  \eps^{-1/2} |\dot{c}_\alpha(s)|  \eps^{5/2} (\dot{\tau}_\alpha - c_\alpha) 
\\ \\ & & + (|\dot{\tau}_\alpha(s) - c_\alpha(s)| + |c_\alpha(s) - c_\alpha(t_0)| + \|v(s)\| + \eps^4 )\|v_{2,\alpha}(s)\|_\alpha + {\tt exp} \ea 
\]
Here the two-sided slowly decaying weighted norm $\| \cdot \|_{W_\alpha}$ is given by
\begin{equation}
 \| x \|_{W_\alpha(t)}^2 := \sum_{k \in \Z} e^{-2\bar{a}|k - \tau_\alpha(t)|} x_k^2,
 \label{eq:Wdef}
 \end{equation}
and $\bar{a} \sim \eps^2$.  The reason for this choice is made clear in the proof of Lemma $\ref{Lem:Virial}$.  

As is often done in stability analysis for PDEs, we first restrict attention to solutions for which $G_\alpha$ above is small.  So long as this holds, after some computation and applications of Gronwall's inequality and Young's inequality, we obtain 
\[ \int_{t_0}^{t_1} \| v_{2,\alpha}(t) \|_\alpha^2 dt \le C \eps^{-2} \int_{t_0}^{t_1} \|v_{1,\alpha}(t)\|_{W_\alpha}^2 dt + C \eps^{-3}(\|v_{1,\alpha}(t_0)\|_{W_\alpha}^2 + \|v_{2,\alpha}(t_0)\|_\alpha^2) + {\tt exp} \]
Since we have followed Mizumachi and chosen $v_1$ to be an exact solution of $\eqref{FPU}$ we may use the virial identity $\eqref{eq:virial}$ to see that $$\int_{t_0}^t \| v_{1,\alpha} (s) \|_{W_\alpha}^2 \le C \eps^{-4} \| v_1(t_0) \|^2 + {\tt exp}.$$
Through a separate analysis, which also mimics \cite{PW, PF2, Mizumachi}, we obtain 
$$\dot{c}_\alpha  \le C(\eps^2 + \|v_{1,\alpha} \| + \|v_{2,\alpha} \|) (\|v_{1,\alpha}\|_{W_\alpha}^2 + \|v_{2,\alpha}\|_\alpha^2) + {\tt exp}.$$ 
and upon integrating this equation and using the above, we see that $|c_\alpha(t) - c_\alpha(t_0)|$ remains small as long as $G_\alpha$ does.  We may now conclude {\it a posteriori} that $G_\alpha$ and hence $|c_\alpha(t) - c_\alpha(t_0)|$ remains small for all time.  In particular, we may use $\eqref{eq:Intro_energy}$ to conclude that $\|v(t)\|$ remains small for all time.    

Note that large negative powers of $\eps$ appear in several of the estimates above.  These large coefficients arise because in the small amplitude long wavelength regime, solitary waves move only slightly faster than smaller amplitude radiation, which disperses at the sonic speed $c = 1$.  Thus, the weighted norms associated with these solitary waves decay very slowly, hence have very large integrals with respect to time.  Given that the large coefficients are present, controlling the above quantities requires some care.

\section{Coordinates for the FPU Flow Near the Sum of Two Solitary Waves}

As discussed in \cite{PF1, PF2} there is a $2$ dimensional manifold of wave states parameterized by wavespeed $c$ and phase $\tau$.  Define the tangent vectors $\xi_{1}(\tau,c) := \partial_x u_{c}(\cdot - \tau)$ and $\xi_{2}(\tau,c) = \partial_c u_c(\cdot-\tau)$ so that $\xi_1(\tau,c)$ and $\xi_2(\tau,c)$ span the tangent  space of the manifold of wave states at the point $u_c(\cdot - \tau)$.  When we evaluate $\xi_i$ at $(\tau_\alpha,c_\alpha)$ we suppress the argument and refer to it as $\xi_{i,\alpha}$. 

In the spirit of \cite{PF2} we define the bilinear forms 
$$\omega_+(x,y) = \sum_{n \in \Z} \left[ y_{1,n}\sum_{k=-\infty}^n x_{2,k} + y_{2,n}\sum_{k=-\infty}^{n-1}x_{1,k}\right] $$
and
$$\omega_-(x,y) = - \sum_{n \in \Z} \left[y_{1,n} \sum_{k=n+1}^\infty x_{2,k} + y_{2,n} \sum_{k=n}^\infty x_{1,k}\right]$$
Observe that $\omega_+(x,y) - \omega_-(x,y) = \sum_{n \in \Z}\sum_{k \in \Z} (y_{1,n} x_{2,k} + y_{2,n}x_{1,k}$) so that whenever either $x$ or $y$ is zero mean, the forms $\omega_+(x,y)$ and $\omega_-(x,y)$ agree.  Also, when either $x$ or $y$ is zero mean, the forms have the anti-symmetry property
\begin{equation} \omega_\alpha(x,y) = -\omega_\alpha(y,x) \qquad . \label{sym1} \end{equation}
For this reason we refer to $\omega_\alpha$ as a symplectic form even though $\eqref{sym1}$ does not hold in general.  A particular case when $x$ is zero mean is when $x$ is in the range of $J$, thus  $\omega_\alpha(Jz,y) = \langle z,y \rangle$.
The reason that we take care to define two separate forms is because we will have occasion to apply them to sequences which needn't be zero mean such as $\xi_{2,\alpha}$.  The key difference between $\omega_+$ and $\omega_-$ is appears in how they act on the weighted spaces $\ell^2_a$ and $\ell^2_{-a}$.  

When $a > 0$ the form $\omega_+$ is continuous when regarded as a map from $\ell^2_{-a} \times \ell^2_a$ to $\R$ but not when regarded as map from $\ell^2_a \times \ell^2_{-a} \to \R$.  Similarly, $\omega_-$ is continuous when regarded as a map from $\ell^2_{a} \times \ell^2_{-a}$ to $\R$, but not when regarded as a map from $\ell^2_{-a} \times \ell^2_a$ to $\R$.   Both $\omega_+$ and $\omega_-$ are only defined on a dense subset of $\ell^2 \times \ell^2 \to \R$ and neither are continuous with respect to the $\ell^2 \times \ell^2$ topology.  The following useful inequality is an immediate consequence of the Cauchy-Schwartz inequality for the $\ell^2_a$, $\ell^2_{-a}$ dual pairing.
\begin{equation}
\omega_\alpha(x, y) \le \frac{1}{1-e^{-a}} \| x \|_{\ell^2_{ - \alpha a}} \| y \|_{\ell^2_{\alpha a}} \sim a^{-1}\| x \|_{\ell^2_{ - \alpha a}} \| y \|_{\ell^2_{\alpha a}}
\label{eq:OmegaBdd}
\end{equation}

In the tradition of Pego and Weinstein \cite{PW}, Pego and Friesecke \cite{PF2} , and Mizumachi \cite{Mizumachi} we impose orthogonality conditions on the perturbation of our ansatz.  We seek to study solutions $u$ of $\eqref{FPU}$ which satisfy 
\begin{equation}
u(k,t) = u_{c_+(t)}(k-\tau_+(t)) + u_{c_-(t)}(k - \tau_-(t)) + v_1(k,t) + v_2(k,t)
\label{eq:Ansatz}
\end{equation}
where $v_1(k,t)$ also satisfies $\eqref{FPU}$.  Let $h_+(k) = \chi_{[0,\infty)}(k)$ be the Heaviside function and $h_-(k) = 1 - h_+(k)$.  Denote $v_{i,\alpha} = v_i h_\alpha$ and $v = v_1 + v_2$.  The orthogonality conditions that we impose on $v_{2,\alpha}$ are 
\begin{eqnarray}
\omega_\alpha(\xi_{1,\alpha},v_\alpha) = 0 \qquad \alpha \in \{+,-\} \label{Perp1}
\\ 
\omega_\alpha(\xi_{2,\alpha},v_{2,\alpha}) = 0 \qquad \alpha \in \{+,-\}.
\label{Perp2}
\end{eqnarray}

\begin{rmk}
Here condition $\eqref{Perp1}$ is imposed on the full perturbation $v$ while condition $\eqref{Perp2}$ is imposed only on $v_2$.  Pego and Friesecke do not decompose the perturbation and require $\omega(\xi_i,v) = 0$.  Mizumachi introduced the analogue of $\eqref{Perp2}$ (for a single wave) which allowed him to control the modulation equation for $c_\alpha$ well enough to obtain orbital stability.
\end{rmk}

Substituting $\eqref{eq:Ansatz}$ into $\eqref{FPU}$ we obtain an evolution equation for $v_{2,\alpha}$.

\begin{equation}
L_\alpha v_{2,\alpha} = (Jg_\alpha)h_\alpha + \tilde{\ell}_\alpha \label{eq:perturb}
\end{equation}
where
\begin{equation}
L_\alpha := \partial_t - JH''(u_{c_\alpha}) \label{eq:Ldef},
\end{equation}
\begin{equation}
\tilde{\ell}_\alpha := \left[\sum_{\beta \in \{+,-\}} \dot{c}_\beta \xi_{2,\beta} - \left(\dot{\tau}_\beta - c_\beta\right) \xi_{1,\beta} \right] h_\alpha \label{eq:ldef}, \mbox{ and } 
\end{equation}
\begin{equation}
g_\alpha := H'(u_{c_+} + u_{c_-} + v_1 + v_2) - H'(v_1) - H'(u_{c_+}) - H'(u_{c_-}) - H''(u_{c_\alpha})v_{2,\alpha}.
\label{eq:gdef}
\end{equation}

Our aim going forward is to analyze the evolution of the quantites $v_2$, $v_1$, $c$ and $\tau$ rather than our solution $u$.  Before we may do this in earnest, we must show that the two are equivalent; i.e. that $(\tau,c,v_1,v_2)$ form a coordinate system for $u$.  Moreover, we must show that these coordinates are valid even if $v_1$ and $v_2$ are as large as the radiation due to collision could be, i.e. as large as the error terms in Theorem $\ref{thm:Finite_Time_Collision}$.

\subsection{Estimates in the KdV regime}

In Theorem $\ref{thm:Orbital_Stability}$ we are careful to demonstrate how the size of both an admissible initial perturbation $\delta$ and the evolution of that perturbation depend on $\eps$.  This is necessary so that we may establish that error incurred from the finite-time approximation yields an initial condition which is not too far from the sum of solitary waves and moreover remains small compared to the main waves.  This requires an understanding of how each constant which appears in \cite{PF4,Mizumachi} depends upon $\eps$.  In this section, we establish small $\eps$ asymptotics for the higher derivatives of the wave profile with respect to $c$ and $\tau$.  Some of the details of the computations appear in the Appendix A.1.  

The small $\eps$ asymptotics for the first derivative of the wave profile are established in \cite{PF1}.  In this section, we refine those estimates.  Upon substituting a wave profile $r(t,k) = r_c(k-ct)$ into the second order differential equation for $r$, $\ddot{r}_k = (S + S^{-1} - 2I)V'(r_k)$, one obtains the differential difference equation $r_c''(\xi) = (T + T^{-1} - 2I) V'(r_c)$ where the translation operator $T$ is given by $(Tx)(\xi) = x(\xi + 1)$.  Taking the Fourier transform one obtains the fixed point equation 
$r_c = P_cN(r_c)$.  Here $N(r) := V'(r) - r = \frac{1}{2}r^2(1 + \eta(r))$ with the above equation serving as the definition of both $N$ and $\eta$.  The operator $P$ is a pseudodifferential operator with symbol $p_c(\xi) = \frac{\sinc^2(\xi/2)}{c^2 - \sinc^2(\xi/2)}$.  One can now make the scaling $c = 1 + \frac{\eps^2}{12}\beta$\footnote{If $c^2 = \frac{\eps^2 \beta'}{12}$, then $\beta' = \beta + 12\eps^2(\frac{\beta}{24})^2$, thus in the small $\eps$ regime the scalings $c^2 = 1 + \frac{\eps^2}{12}\beta$ and $\pm c = 1 + \frac{\eps^2 \beta}{24}$ are equivalent to leading order in $\eps$.} and the renormalization $r_c(\cdot) \mapsto \phi^{(\eps,\beta)} := \eps^2 r_c(\frac{\cdot}{\eps})$ to obtain a sequence of fixed point equations
$\phi^{(\eps,\beta)} = P^{\eps} N^{\eps} (\phi^{(\eps,\beta)})$.  Here $P^{(\eps,\beta)}$ has the symbol $p^{(\eps)}(\xi) = \frac{\eps^2 \sinc^2 (\eps \xi/2) }{1 + \frac{\eps^2}{12} \beta - \sinc^2(\eps \xi/2)}$ and $N^\eps(\phi) := \eps^{-4}N(\eps^2 \phi) = \frac{1}{2}\phi^2(1 + \eta(\eps^2 \phi))$.  Note that whenever $p$ appears with a superscript, it denotes the symbol of a pseudodifferential operator whereas $p$ with a subscript is used to denote the momentum component of a solution to $\eqref{FPU}$.

One may now take the limit as $\eps \to 0$ and recover the fixed point equation for the KdV wave profile $\phi_\beta = P^{(0)} N^{(0)} (\phi_\beta)$ where $P^{(0)}$ has symbol $p^{(0)}(\xi) = \frac{12}{\xi^2 + \beta}$ and $N^{(0)}(\phi) = \frac{1}{2} \phi^2$.  The crux of the argument in \cite{PF1} is that $p^\eps \to p^0$ uniformly on $\R$.  It is then an elementary result in the theory of pseudodifferential operators that $P^{(\eps,\beta)} \to P$ in the operator norm.  This (together with the invertibility of $(I - P^0 N^0(\phi_\beta))$ on a weighted space orthogonal to the neutral modes which is established from the theory of the KdV equation) is precisely what is needed to show that the fixed point equation $\phi^{(\eps,\beta)} = P^{(\eps,\beta)} N^\eps (\phi^{(\eps,\beta)})$ can be solved via a contraction mapping argument.

In this paper, the collision problem requires precise small $\eps$ asymptotics for $\phi^{(\eps,\beta)}$ and its derivatives in the weighted space $H^1_a$.  We accomplish this by regarding the fixed point equation $\phi^{(\eps,\beta)} = P^{(\eps,\beta)} N^\eps(\phi^{(\eps,\beta)})$ in the space $E^3_a$ of even functions whose derivatives up to order $3$ are square integrable after multiplication by the weight $e^{ax}$.  To solve this equation via contraction mapping arguments we must now obtain small $\eps$ asymptotics for $P^{(\eps,\beta)} - P^0$ as well as $(I - P^0 N^0(\phi_\beta))^{-1}$, regarded as operators in the weighted space.  To that end, we present the following lemma which is proven in Appendix A.1.

\begin{lemma} \label{lem:small_eps_1}
There are constants $C > 0$ and $a > 0$ such that for $\eps$ sufficiently small and any $s \ge 0$ the following hold
\begin{enumerate}
\item $\| P^{(\eps,\beta)} - P^{(0)} \|_{\mathcal{L}(H^s_a)} \le C \eps^2$,
\item $\| \partial_\beta^k P^{(\eps,\beta)} - \partial_\beta^k P^{(0)} \|_{\mathcal{L}(H^s_a)} \le C \eps^2$
\item $\| \partial^k N^\eps (x) - \partial^k N^0(y) \|_{\mathcal{L}(H^s_a)} \le \| x - y\|_{H^s_a} + K\eps^2$ for $k = 0,1,2$,
\item $ \| \left[I - P^{(0)} \partial N^{(0)} (\phi_\beta) \right]^{-1} \|_{\mathcal{L}(E^s_a)} \le C$,

\end{enumerate}
\end{lemma}

Here, and always, when $X$ is a Banach space, $\mathcal{L}(X)$ denotes the Banach space of bounded linear operators from $X$ to itself and $\| \cdot \|_{\mathcal{L}(X)}$ is the operator norm.  With the aid of Lemma $\ref{lem:small_eps_1}$ we may establish the leading order behavior for small $\eps$ of the higher derivatives of the wave profile with respect to $c$ and $\tau$.  

\begin{lemma} \label{lem:small_eps_2}
There is a constant $C$ which does not depend on $\eps$ such that for $\eps$ sufficiently small
\begin{equation}
\| \partial^k_\beta \phi^{(\eps,\beta)} - \partial^k_\beta \phi_\beta \|_{H^3_{a/\eps}} \le C \eps^2
\label{eq:eps^2}
\end{equation}
and
\begin{equation}
\| \partial_c^k \partial_\tau^j r_c \|_{L^2_{\eps a}} + \| \partial_c^k \partial_\tau^j p_c \|_{L^2_{a}} \le C \eps^{3/2 +j - 2k}
\label{eq:power_rules}
\end{equation}
hold with $k + j = 0,1,2,3$.
\end{lemma}
\begin{rmk}
The weight $a/\eps$ which appears in the equation for the renormalized wave profiles $\eqref{eq:eps^2}$ aught to be regarded as the renormalized version of the weight $a$.  Recall that $a \sim \eps$ so $a/\eps \sim 1$.
\end{rmk}
\begin{proof}
We first show that $\eqref{eq:power_rules}$ follows from $\eqref{eq:eps^2}$.  Suppose that $\eqref{eq:eps^2}$ holds.  Then since $\partial^k_\beta \phi_\beta$ doesn't depend on $\eps$ it follows that $\partial_\beta^k \partial_\xi^j \phi^{(\eps,\beta)} = \mathcal{O}(1)$ in the small $\eps$ regime for $j$ and $k$ between $0$ and $3$, in particular for $0 \le j+k \le 3$.  To obtain the bound on $r_c$ and its derivatives in $\eqref{eq:power_rules}$ one merely uses the definition $r_c(x,\tau) = \eps^2 \phi^{(\eps,\beta)}(\eps (x-\tau))$, the observation that $\frac{d \beta}{dc} = \mathcal{O}(\eps^{-2})$ for the scaling $c^2 = 1 + \frac{\eps^2}{12}\beta$, the observation that $\partial_\tau r_c(x,\tau) = \eps^3 \partial_x \phi^{(\eps,\beta)}(\eps(x-\tau))$, and the fact that $\| f(\frac{\cdot}{\eps}) \|_{L^2_a} = \eps^{1/2} \| f(\cdot) \|_{L^2_{\eps a}}$.  To obtain the bound on $p_c$, write 
\begin{equation}
r_c(\frac{x}{\eps}) = q_c(\frac{x}{\eps} + 1) - q_c(\frac{x}{\eps}) = q_c'(\frac{x + \eps \eta}{\eps}) = \frac{-1}{c} p_c(\frac{x + \eps \eta}{\eps})
\label{eq:rc_pc}
\end{equation} 
for some $\eta \in (0,1)$ which may depend upon $x$.  Here the second equality uses the mean value theorem and the last equality uses the fact that wave profiles satisfy the FPU equation to relate spatial and temporal derivatives.
Now use the fact that for any $\eta \in (0,1)$ we have $p_c(\xi + \eta) < p_c(\xi +1) + p_c(\xi)$ to conclude that $\| p_c \|_{L^2_a} < 2c\|r_c \|_{L^2_a}$.  Similarly, we can take derivatives with respect to $c$ or $\tau$ in $\eqref{eq:rc_pc}$ to establish the bound on $p_c$ from the bound on $r_c$ in $\eqref{eq:power_rules}$.

The rest of the proof consists of establishing that $\eqref{eq:eps^2}$ follows from Lemma $\ref{lem:small_eps_1}$.
We first solve the fixed point equation $\phi^{(\eps,\beta)} = P^{(\eps,\beta)} N^\eps(\phi^{(\eps,\beta)})$ in the weighted space of even functions $E^3_a$.  A standard uniform contraction principle approach allows one to continue the fixed point at $\eps = 0$ to nonzero values of $\eps$ so long as the equation is sufficiently smooth in $\eps$, and the linear part of the equation at the base point $\phi_\beta$ is invertible.  Moreover, it is a consequence of the uniform contraction principle that the fixed point $\phi^{\eps,\beta}$ is as smooth in $\beta$ as the operator $P^{\eps,\beta}$ and further that the fixed point $\phi^{\eps,\beta}$ is close to the fixed point $\phi_\beta$ when $\eps$ is close to $0$.  More quantitatively, $\| \phi^{\eps,\beta} - \phi_\beta\| \le \| P^{(\eps,\beta)} N^\eps - P^0 N^0 \|$.  The details of this argument are written out for the $C^1$ case in Lemma A.1 in \cite{PF1}.  Here we apply the $C^3$ uniform contraction principle, rather than the $C^1$ uniform contraction principle, but otherwise the arguments are the same.

Observe that the operator $I - P^0 N^0(\phi_\beta)$ is a compact perturbation of a pseudodifferential operator, thus its essential spectrum in $H^s$ is independent of $s$.  Moreover, from \cite{PW} it's essential spectrum in $E^0_a$ is bounded away from the imaginary axis and it has no even eigenfunctions in $L^2$, thus no eigenfunctions in $E^s_a$ for any $s \ge 0$.  Thus $\| (I - P^0 N^0(\phi_\beta))^{-1}\|_{\mathcal{L}(E^3_a)} \le C$.  From Lemma $\ref{lem:small_eps_1}$ we have  $\| P^{(\eps,\beta)} - P^0 \|_{\mathcal{L}(E^1_a)} \le C\eps^2$.  Thus we have the requisite regularity in $\eps$ and invertibility of the linear part to apply the argument.  We thus obtain for $\eps$ sufficiently small:
\begin{enumerate}
\item that $\phi^{(\eps,\beta)}  \in E^3_a$,
\item that $\| \phi^{(\eps,\beta)} - \phi_\beta \| \le C \eps^2$, 
\item that $\| (I - P^{(\eps,\beta)} N^\eps(x))^{-1} \|_{\mathcal{L}(E^3_a)}$ is bounded uniformly for $\| x - \phi_\beta\|$ small, and
\item that $\beta \mapsto \phi^{\eps,\beta}$ is thrice continuously differentiable.
\end{enumerate}

One may now implicitly differentiate the fixed point equations 
$\phi^{(\eps,\beta)} = P^{(\eps,\beta)} N^\eps (\phi^{(\eps,\beta)})$ and $\phi_\beta = P^0 N^0(\phi_\beta)$, isolate the difference $ \partial_\beta^k \phi^{(\eps,\beta)} - \partial_\beta^k \phi_\beta$, and bound the terms on the right hand side.
In the case $k = 1$ we obtain:
$$\ba{lll}
\partial \phi^{(\eps,\beta)} - \partial \phi_\beta & = & (I - P^{(\eps,\beta)} \partial N^\eps(\phi^{(\eps,\beta)}))^{-1} \left( (P^0\partial N^0(\phi_\beta) - P^{(\eps,\beta)} \partial N^\eps(\phi^{(\eps,\beta)})) \partial_\beta \phi_\beta \right. \\ \\
& & \qquad + \left. \partial_\beta P^{(\eps,\beta)} N^\eps(\phi^{(\eps,\beta)}) - \partial_\beta P^0 N^0(\phi_\beta)\right). \ea $$
Thus in light of $(2)$ above and Lemma $\ref{lem:small_eps_1}$ we see (after adding and subtracting relevant terms and using the triangle inequality) that 
\begin{equation}
\|\partial_\beta \phi^{(\eps,\beta)} - \partial_\beta \phi_\beta\|_{H^s_a} \le C\eps^2
\label{eq:k1}
\end{equation}
as desired.

Implicitly differentiating again we obtain
$$ \partial^2_\beta \phi^{(\eps,\beta)} - \partial^2_\beta \phi_\beta = (I - P^{(\eps,\beta)} \partial N^\eps(\phi^{(\eps,\beta)}))^{-1} ( (i) + (ii) + (iii) + (iv))$$ where
$(i) = (P^0 \partial N^0 (\phi_\beta) - P^{(\eps,\beta)} \partial N^\eps(\phi^{(\eps,\beta)}))\partial_\beta^2 \phi_\beta$,
$(ii) = \partial_\beta^2 P^{(\eps,\beta)} N^\eps(\phi^{(\eps,\beta)}) - \partial_\beta^2 P^0 N^0(\phi_\beta)$,
$(iii) = 2\partial_\beta P^{(\eps,\beta)} \partial N^\eps(\phi^{(\eps,\beta)})\partial_\beta \phi^{(\eps,\beta)} - 2\partial_\beta P^0 \partial N^0(\phi_\beta)\partial_\beta \phi_\beta$, and
$(iv)  = P^{(\eps,\beta)} \partial^2 N^\eps(\phi^{(\eps,\beta)})(\partial_\beta \phi^{(\eps,\beta)})^2 - P^0 \partial^2 N^0(\phi_\beta)(\partial_\beta \phi_\beta)^2$.

We may add and subtract relevant terms, use the triangle inequality, and bound each term with either  $(2)$ above, equation $\eqref{eq:k1}$, or Lemma $\ref{lem:small_eps_1}$ to see

\begin{equation}
 \| \partial^2_\beta \phi^{(\eps,\beta)} - \partial^2_\beta \phi_\beta \|_{H^s_a} \le C \eps^2
 \label{eq:k2}
 \end{equation}
as desired.

Implicitly differentiating once more, we obtain

$$\partial^3_\beta \phi^{(\eps,\beta)} - \partial_\beta^3 \phi_\beta = (I - P^{(\eps,\beta)} \partial N^\eps(\phi^{(\eps,\beta)}))^{-1}( (i) + (ii) + (iii) + (iv) + (v) + (vi) + (vii))$$
with
$(i) = [P^0 \partial N^0(\phi_\beta) - P^{(\eps,\beta)} \partial N^\eps (\phi^{(\eps,\beta)})]\partial_\beta^2 \phi_\beta$,
$(ii) = \partial_\beta^3 P^{(\eps,\beta)} N^\eps(\phi^{(\eps,\beta)}) - \partial_\beta^3 P^0 N^0(\phi_\beta)$,
$(iii) = 3\partial_\beta^2 P^{(\eps,\beta)} \partial N^\eps(\phi^{(\eps,\beta)})\partial_\beta \phi^{(\eps,\beta)} - 3 \partial_\beta^2 P^0 \partial N^0 (\phi_\beta) \partial_\beta \phi_\beta$,
$(iv) = 3\partial_\beta P^{(\eps,\beta)} \partial^2 N^\eps(\phi^{(\eps,\beta)}) (\partial_\beta \phi^{(\eps,\beta)}, \partial_\beta \phi^{(\eps,\beta)}) - 3\partial_\beta P^0 \partial^2 N^0 (\phi_\beta)(\partial_\beta \phi_\beta, \partial_\beta \phi_\beta)$,
$(v) = 3\partial_\beta P^{(\eps,\beta)} \partial N^\eps(\phi^{(\eps,\beta)}) \partial_\beta^2 \phi_\beta - 3\partial_\beta P^0 \partial N^0 (\phi_\beta) \partial_\beta^2 \phi_\beta$, \\
$(vi) = P^{(\eps,\beta)} \partial^3 N(\phi^{(\eps,\beta)})(\partial_\beta \phi^{(\eps,\beta)}, \partial_\beta \phi^{(\eps,\beta)}, \partial_\beta \phi^{(\eps,\beta)}) - P^0 \partial^3 N(\phi^0)((\partial_\beta \phi_\beta, \partial_\beta \phi_\beta, \partial_\beta \phi_\beta)$, and
$(vii) = 3 P^{(\eps,\beta)} \partial^2 N^\eps(\phi^{(\eps,\beta)})(\partial_\beta \phi^{(\eps,\beta)}, \partial_\beta^2 \phi^{(\eps,\beta)}) - 3 P^0 \partial^2 N^0(\phi^0)(\partial_\beta \phi^0, \partial_\beta^2 \phi^0)$.

Once again, we may add and subtract relevant terms, use the triangle inequality, and bound each term with either  $(2)$ above, equation $\eqref{eq:k1}$, equation $\eqref{eq:k2}$ or Lemma $\ref{lem:small_eps_1}$ to see

$$ \| \partial^3_\beta \phi^{(\eps,\beta)} - \partial^3_\beta \phi_\beta \|_{H^s_a} \le C \eps^2$$
as desired.  This completes the proof.

\end{proof}

We now use these small $\eps$ asymptotics to establish the existence of coordinates in a tubular neighborhood of the sum of two well separated solitary waves.

\subsection{Tubular Coordinates}
For the remainder of the paper we will regard $(\tau,c,v_1,v_2)$ as coordinates for the solution $u$.  Here the pair $(\tau,c)$ is shorthand for the four-tuple $(\tau_+,\tau_-,c_+,c_-)$.  The strategy for proving Theorem $\ref{thm:Orbital_Stability}$ is as follows:
\begin{enumerate}
\item Let $u$ be the solution of $\eqref{FPU}$ of interest.  Pick an initial time $t_0$, initial phases $\tau_\alpha^*$, initial speeds $c_\alpha^*$, and initial localized perturbations $v_{2,\alpha}^*$ and let $v_1$ solve $\eqref{FPU}$ with $v_1(t_0) = u(t_0) - \sum_{\alpha} (u_{c_\alpha^*(\cdot - \tau_\alpha^*)} + v_{2,\alpha}^*)$

\item Given $t > t_0$ determine the unique $\tau_\alpha = \tau_\alpha(t)$ and $c_\alpha = c_\alpha(t)$ such that 
$$\omega_\alpha(\xi_{1,\alpha}, (u(t) - \sum_{\alpha} u_{c_\alpha}(\cdot - \tau_\alpha))h_\alpha) = \omega_\alpha(\xi_{2,\alpha}, (u(t) - v_1(t) - \sum_\alpha u_{c_\alpha}(\cdot - \tau_\alpha))h_\alpha) = 0.$$
Define $v_2(t) := u(t) - v_1(t) - \sum_{\alpha} u_{c_\alpha(t)}(\cdot - \tau_\alpha(t))$.

\item Derive equations for the evolution of $\tau_\alpha$, $c_\alpha$, and $v_2$ 

\item Use the equations derived in $(3)$ to control the evolution of $\tau_\alpha$, $c_\alpha$, and $v_2$.

\item Given the coordinates $(\tau,c,v_2)$ constructed in $(2)$ and controlled in $(4)$, conclude that $u(t) = \sum_{\alpha} u_{c_\alpha(t)} (\cdot - \tau_\alpha(t)) + v_2(t) + v_1(t)$ is well-behaved.

\end{enumerate}

In our application of Theorem $\ref{thm:Orbital_Stability}$ to the proof of Theorem $\ref{thm:Collision}$, the choice of $t_0$, $\tau_\alpha^*$, $c_\alpha^*$, and $v_{2,\alpha}^*$ in step $(1)$ are naturally given by Theorem $\ref{thm:Finite_Time_Collision}$.  Step 3 is the subject of section 3.3 and step 4 is the subject of section 4.  Step 5 follows trivially from step 4.  This section is concerned with step 2.  Similar work is done in section 2 of \cite{PF2} and Lemma 4 of \cite{Mizumachi}, though for a single wave.  A key difference is that we must once again control the size with respect to $\eps$ of the neighborhood on which this decomposition is valid.

First, we state and prove the following lemma, which provides some preliminary estimates on the restriction of the symplectic forms to the tangent space of the wave state manifold.  It correponds to Lemma 2.1 in \cite{PF2}.  In what follows the matrix $A$ is given by

\begin{equation} A(\tau,c) = A_0(\tau,c) + A_1(\tau,c) \label{E00} \end{equation} where

$$A_0(\tau,c) = \left(\ba{cc} A_{0+}(\tau_+,c_+) & 0 \\ \\ 0 & A_{0-}(\tau_-,c_-) \ea \right) \mbox{ with }
A_{0\alpha}(\tau_\alpha,c_\alpha) = \left(\ba{cc} \omega_\alpha(\xi_{1,\alpha},\xi_{1,\alpha}h_\alpha) & \omega_\alpha(\xi_{1,\alpha},\xi_{2,\alpha}h_\alpha) \\ \\
\omega_\alpha(\xi_{2,\alpha},\xi_{1,\alpha}h_\alpha) & \omega_\alpha(\xi_{2,\alpha},\xi_{2,\alpha}h_\alpha \ea \right),$$

$$A_1(\tau,c) = \left(\ba{cc} 0 & A_{1+}(\tau_+) \\ \\ A_{1-}(\tau_-) & 0 \ea \right)
\mbox{ with }
A_{1\alpha}(\tau_\alpha,c_\alpha) = \left(\ba{cc} \omega_\alpha(\xi_{1,\alpha},\xi_{1,-\alpha}h_\alpha) & \omega_i(\xi_{1,\alpha},\xi_{2,-\alpha}h_\alpha) \\ \\ \omega_i(\xi_{2,\alpha},\xi_{1,-\alpha}h_\alpha) & \omega_\alpha(\xi_{2,\alpha},\xi_{2,-\alpha} h_\alpha) \ea\right).$$ 

Here as always $\xi_{i,\alpha}$ is evaluated at $(\tau_\alpha,c_\alpha)$.

\begin{rmk}
In the following lemma, and in the remainder of the paper, we use the notation ${\tt exp}$ to denote terms $x$ which satisfy $|x| \le C (e^{-\eps m |\tau_+|} + e^{-\eps m |\tau_-|})$ for some constants $C$ and $m$ which may be chosen uniformly in $\eps$.  When the value of $\tau_\alpha$ is allowed to vary we write ${\tt exp(s)}$ to denote terms $x$ which satisfy $|x| \le C(e^{-\eps m | \tau_+(s)|} + e^{-\eps m |\tau_-(s)|})$.  Notice that for any  $p$ we have $\eps^p {\tt exp} = {\tt exp}$ by making $m$ smaller and $C$ larger if necessary.  This convention will simplify many expressions throughout the paper.  
\end{rmk}

\begin{lemma} \label{L6}
There is a positive constant $T_0$ such that if $\tau_+, - \tau_- \ge T_0$ then the matrix $A = A(\tau,c,w)$ given by $\eqref{E00}$ is invertible.  Furthermore, the off-diagonal terms $A_1$ satisfy $\|A_1\| = {\tt exp}$.  Moreover, in the KdV scaling $\eqref{eq:KdV_approx}$ to leading order in $\eps$ we have $T_0 \sim \frac{|\log \eps|}{\eps}$ and 
\begin{equation}
A^{-1} \sim
\left( \ba{cccc} 
\eps^{-4} & \eps^{-1} & {\tt exp} & {\tt exp} \\ 
\eps^{-1} & {\tt exp} & {\tt exp} & {\tt exp} \\
{\tt exp} & {\tt exp} & \eps^{-4} & \eps^{-1} \\
{\tt exp} & {\tt exp} & \eps^{-1} & {\tt exp}
\ea \right)
\label{eq:Ainv_small_eps}
\end{equation}
\end{lemma}

\begin{rmk}
When $A$ and $B$ are matrices and we write $A \sim B$, this means that there are positive constants $c_{ij}$ and $C_{ij}$ which don't depend on $\eps$ such that $c_{ij} b_{ij} \le  a_{ij} \le C_{ij} b_{ij}$ holds.
\end{rmk}

\begin{proof}
We show that $A_0$ is the sum of a constant invertible matrix, plus a term which is exponentially small in the phases $\tau_+$ and $\tau_-$ which measure the separation between the solitary waves $u_{c_\alpha}$ and the cutoff of the localization $h_\alpha$.  We also show that the off-diagonal terms $A_1$ go to zero exponentially in the phases $\tau_+$ and $\tau_-$.  This completes the proof.  The details follow.

The matrices $A_{0,\alpha}$ were studied in Lemma 2.1 of \cite{PF2} for the case $\alpha = +$ and $h_+ \equiv 1$.  Our result is similar though due to the presence of the non-constant function $h_\alpha$ the matrices contain exponentially small, time dependent components.  Furthermore, the collision problem demands that we keep careful track of the $\eps$-dependence of all quantities, unlike \cite{PF2}.  Consider the matrix
$$
\ba{lll}
A_{0,+}(\tau_+,c_+) & = & \left( \ba{cc} \omega_+(\xi_{1,+},\xi_{1,+}h_+) & \omega_+(\xi_{1,+},\xi_{2,+}h_+) \\ \\ \omega_+(\xi_{2,+},\xi_{1,+}h_+) & \omega_+(\xi_{2,+},\xi_{2,+}h_+) \ea \right)  \\ \\
& = & \left( \ba{cc} \omega_+(\xi_{1,+},\xi_{1,+}) & \omega_+(\xi_{1,+},\xi_{2,+}) \\ \\ \omega_+(\xi_{2,+},\xi_{1,+}) & \omega_+(\xi_{2,+},\xi_{2,+}) \ea \right)  - \left( \ba{cc} \omega_+(\xi_{1,+},\xi_{1,+}h_-) & \omega_+(\xi_{1,+},\xi_{2,+}h_-) \\ \\ \omega_+(\xi_{2,+},\xi_{1,+}h_-) & \omega_+(\xi_{2,+},\xi_{2,+} h_-) \ea \right) \\ \\
& := & A_0^{FP} - \Delta_{0,+}(\tau_+,c_+).
\ea
$$
Here $A_0^{FP}$ is exactly the matrix studied in Lemma 2.1 of \cite{PF2}.  From that lemma, and from Lemma 9.1 of \cite{PF1} it follows both that $A_0^{FP}$ is constant and invertible, and that $A_0^{FP} \sim \left( \ba{cc} 0 & \eps \\ \eps & \eps^{-2} \ea \right)$; thus $(A_0^{FP})^{-1} \sim \left( \ba{cc} \eps^{-4} & \eps^{-1} \\ \eps^{-1} & 0 \ea \right)$. For the matrix $\Delta_{0,+}$, note that due to the inequality $\eqref{eq:OmegaBdd}$ we see that $\omega_\alpha(\xi_{i,\alpha},\xi_{j,\alpha} h_{-\alpha}) \le K\| \xi_{i,\alpha} \|_{\ell^2_{-\alpha a}}  \| \xi_{j,\alpha} h_{-\alpha} \|_{\ell^2_{\alpha a}}$.  In light of the fact that $|\xi_{j,\alpha}(k)| \le Ke^{-a| k - \tau_\alpha|}$ we see that $$\| \xi_{j,+} h_{-} \|_{\ell^2_a}^2 = \sum_{k = -\infty}^{-1} Ke^{-2a |k - \tau_+|} \le \frac{K}{a} e^{-2a |\tau_+|} = {\tt exp}$$
In the last inequality we have used the fact that $\tau_+ > 0$.  The factor of $a$ in the denominator comes from summing the exponential.  Thus we obtain that $A_{0,+}$ is the sum of a constant, invertible matrix plus a perturbation that is exponentially small in $|\tau|$.  The inequality $ \|\xi_{j,-}h_+\|_{\ell^2_{-a}} \le \frac{K}{a} e^{-a |\tau_-|}$ also holds, thus we obtain an analogous result for $A_{0,-}$.  In fact, these bounds suffice to yield $\|A_1\| \le \frac{K}{a} (e^{-a \tau_+} + e^{a \tau_-})$.  In particular, denoting $A^{FP} = \left( \ba{cc} A^{FP}_0 & 0 \\ 0 & A^{FP}_0 \ea \right)$ we obtain 
$$ \| A - A^{FP} \| < \frac{K}{a} e^{-a T_0}$$
for $|\tau_\alpha| > T_0$.  Since $A^{FP}_0$ is invertible, it follows that $A$ is also invertible as long as $\frac{K}{a} e^{-a T_0}\|(A^{FP})^{-1}\| < 1$.  Isolating $T_0$ we see that $A$ is invertible for 
\begin{equation}
T_0 > \frac{1}{a}\log\left(\frac{a}{K\|(A^{FP})^{-1}\|}\right) \sim \eps^{-1} \log(\eps^5) \sim \eps^{-1} \log(\eps)
\label{eq:T_0def}
\end{equation}

Since $A^{-1}$ differs from $(A^{FP})^{-1}$ only by terms of the form {\tt exp}, $\eqref{eq:Ainv_small_eps}$ holds and the proof is complete.

\end{proof}

The following proposition quantifies the size of the neighborhood of two well separated solitary waves on which the symplectic orthogonality conditions $\eqref{Perp1}$ and $\eqref{Perp2}$ uniquely specify the modulation parameters $c$ and $\tau$ in the ansatz $\eqref{eq:Ansatz}$.  Its proof is given in Appendix A.3.

\begin{prop} \label{Ptube}
Let $u$ and $v_1$ in $\ell^2$ be given.
Define 
\[ \hat{u}(\tau_+,\tau_-,c_+,c_-) := u_{c_+}(\cdot - \tau_+) + u_{c_-}(\cdot - \tau_-).
\]   For each $\eta > 0$ there is a $C > 0$ which may be chosen uniformly in $\eps$ such that whenever 
\begin{equation} \| v_1 \| + \|(u - \hat{u}(\tau^*_+,\tau^*_-,c^*_+,c^*_-) - v_1)h_\alpha\|_{\ell^2_\alpha} < C\eps^{5/2 + \eta} \qquad t \in [t_0,t_1] \label{eq:usmall}
\end{equation}
for some $\tau^*_- < -T_0 < T_0 < \tau^*_+$ and $c^*_- \sim -1 - \eps^2 < 1 + \eps^2 \sim c^*_+$,
then there is a unique choice of phase $\tau_\alpha$ and speed $c_\alpha$ such that 
\[
\omega_\alpha(\xi_{1,\alpha}, (u - \hat{u}(\tau_\alpha,c_\alpha)h_\alpha) = \omega_\alpha(\xi_{2,\alpha}, (u - \hat{u}(\tau_+,\tau_-,c_+,c_-) - v_1)h_\alpha) = 0
\]
holds.  In particular, upon denoting $v_2(t) = u - v_1 - \hat{u}(\tau_+,\tau_-,c_+,c_-)$ the orthogonality conditions $\eqref{Perp1}$ and $\eqref{Perp2}$ are satisfied.
\end{prop}

\subsection{The Modulation Equations}

Having established that studying the flow of FPU near the sum of well-separated solitary waves is equivalent to studying the coordinates $(\tau,c,v_2)$ we return to equation $\eqref{eq:perturb}$ which describes the evolution of $v_{2,\alpha}$.  To completely describe the evolution of the coordinates $(\tau,c,v_2)$, we must derive evolution equations for the modulation parameters $\tau$ and $c$.  In order to work with modulation parameters whose variation is small, we define $\gamma_\pm$ by $\tau_\pm(t) = \int_{t_0}^t c_\pm(s) ds + \gamma_\pm(t)$ and we study the evolution of $\gamma$ and $c$.
Modulation equations are obtained by applying $\omega_\alpha(\xi_{i,\alpha},\cdot)$ to $\eqref{eq:perturb}$ and moving all terms which do not depend upon $\dot{c}$ or $\dot{\gamma}$ to the right hand side.  The following lemma makes this precise.

\begin{lemma} \label{perp4}
Let $t_0 < t_1$ be real numbers and suppose that the ansatz $\eqref{eq:Ansatz}$ and orthogonality conditions $\eqref{Perp1}$ and $\eqref{Perp2}$ are valid for $t \in [t_0,t_1]$.  Then
\begin{equation}
\omega_\alpha(\xi_{2,\alpha},L_\alpha v_{2,\alpha}) = N_{2,\alpha} + \ell_{2,\alpha}, 
 \label{eq:DiffPerp1}
 \end{equation}
 and
 \begin{equation}
\omega_\alpha(\xi_{1,\alpha},L_\alpha v_{2,\alpha}) = N_{1,\alpha} + \ell_{1,\alpha} \label{eq:DiffPerp2}
\end{equation}
hold for $t \in (t_0,t_1)$.  Here $\xi_{i,\alpha}$ is evaluated at $(\tau_\alpha(t),c_\alpha(t))$, $L_\alpha v_{2,\alpha}$ is evaluated at $t$, and the quantities $\ell_{i,\alpha}$ and $N_{i,\alpha}$ are given by \begin{equation} 
\ell_{1,\alpha} = -\dot{\gamma}_\alpha \omega_\alpha(\partial_\tau \xi_{1,\alpha},v_\alpha) - \dot{c}_\alpha \omega_\alpha( \partial_c \xi_{1,\alpha}, v_\alpha)
\label{eq:ell1_alpha}
\end{equation}
\begin{equation}
\ell_{2,\alpha} = -\dot{\gamma}_\alpha \omega_\alpha(\partial_\tau \xi_{2,\alpha},v_{2,\alpha}) - \dot{c}_\alpha \omega_\alpha(\partial_c \xi_{2,\alpha}, v_{2,\alpha})
\label{eq:ell2_alpha}
\end{equation}
and
\begin{equation}
N_{1,\alpha} = - \omega_\alpha(\xi_{1,\alpha}, L_\alpha v_{1,\alpha}), \qquad 
N_{2,\alpha} = -\omega_\alpha(\xi_{1,\alpha},v_{2,\alpha}).
\label{eq:N_alpha}
\end{equation}

\end{lemma}

\begin{proof}
The proof proceeds by implicitly differentiating
$$\omega_\alpha( \xi_{1,\alpha}(\tau_\alpha(t),c_\alpha(t)), v_\alpha(t)) \equiv \omega_\alpha(\xi_{2,\alpha}(\tau_\alpha(t),c_\alpha(t)), v_{2,\alpha}(t)) \equiv 0$$ to obtain
\begin{equation}  
\dot{\tau}_\alpha \omega_\alpha(\partial_\tau \xi_{1,\alpha}, v_\alpha) + \dot{c}_\alpha \omega_\alpha (\partial_c \xi_{1,\alpha}, v_\alpha) + \omega_\alpha(\xi_{1,\alpha}, \partial_t v_\alpha) \equiv 0 
\label{eq:lemdd1}
\end{equation}\
and
\begin{equation}
\dot{\tau}_\alpha \omega_\alpha(\partial_\tau \xi_{2,\alpha},v_{2,\alpha}) + \dot{c}_\alpha \omega_\alpha(\partial_c \xi_{2,\alpha}, v_{2,\alpha}) + \omega_\alpha(\xi_{2,\alpha}, \partial_t v_{2,\alpha}) \equiv 0 
\label{eq:lemdd2}
\end{equation}
Use the fact that $\partial_\tau \xi_{2,\alpha} = \frac{1}{c_\alpha} JH''(u_{c_\alpha})\xi_{2,\alpha} - \frac{1}{c_\alpha} \xi_{1,\alpha}$, that $\omega_\alpha(Jx,y) = -\omega_\alpha(x,Jy)$, and that $\dot{\tau}_\alpha = c_\alpha + \dot{\gamma}_\alpha$ to rewrite $\eqref{eq:lemdd2}$ as 
$$\omega_\alpha(\xi_{2,\alpha},L_\alpha v_{2,\alpha}) - \omega_\alpha(\xi_{1,\alpha},v_{2,\alpha}) + \dot{\gamma} \omega_\alpha(\partial_\tau \xi_{2,\alpha}, v_{2,\alpha}) + \dot{c}_\alpha \omega_\alpha(\partial_c \xi_{2,\alpha}, v_{2,\alpha}) = 0$$
which is exactly $\eqref{eq:DiffPerp1}$.  Similarly, use the fact that $\partial_\tau \xi_{1,\alpha} = \frac{1}{c_\alpha} JH''(u_{c_\alpha})\xi_{1,\alpha}$, that $\omega_\alpha(Jx,y) = -\omega_\alpha(x,Jy)$, and that $\dot{\tau}_\alpha = c_\alpha + \dot{\gamma}_\alpha$ to obtain
$$ \omega_\alpha(\xi_{1,\alpha},L_\alpha v_\alpha) + \dot{\gamma}_\alpha \omega_\alpha(\partial_\tau \xi_{1,\alpha},v_\alpha) + \dot{c}_\alpha \omega_\alpha(\partial_c \xi_{1,\alpha},v_\alpha) = 0$$ which is exactly $\eqref{eq:DiffPerp2}$.
\end{proof}

With the aim of using Lemma $\ref{perp4}$ to obtain equations for $\dot{c}_\alpha$ and $\dot{\gamma}_\alpha$, define $\tilde{\ell}_{i,\alpha} := \omega_\alpha(\xi_{i,\alpha},\tilde{\ell}_\alpha)$ and 
$\tilde{N}_{i,\alpha} := \omega_\alpha(\xi_{i,\alpha},(Jg_\alpha)h_\alpha)$ where $\tilde{\ell}_\alpha$ and $g_\alpha$ are defined in $\eqref{eq:ldef}$ and $\eqref{eq:gdef}$ respectively
so that upon applying $\omega_\alpha(\xi_{i,\alpha},\cdot)$ to equation $\eqref{eq:perturb}$ and using Lemma $\ref{perp4}$ we obtain a system of four ordinary differential equations for $\dot{c}_\alpha$ and $\dot{\gamma}_\alpha$ which we may write in components as
$$\tilde{\ell}_{i,\alpha} - \ell_{i,\alpha} = N_{i,\alpha} - \tilde{N}_{i,\alpha} \qquad i = 1, \cdots 4,$$
and as a system 
\begin{equation}
\left( \ba{cc} A_{0+} - B_+ & A_{1+} \\ \\ A_{1-} & A_{0-} - B_- \ea \right) \left( \ba{c} \dot{\gamma}_+ \\ \dot{c}_+ \\ \dot{\gamma}_- \\ \dot{c}_- \ea \right) = \left( \ba{c} N_{1,+} - \tilde{N}_{1,+} \\ N_{2,+} - \tilde{N}_{2,+} \\ N_{1,-} - \tilde{N}_{1,-} \\ N_{2,-} - \tilde{N}_{2,-} \ea \right).
\label{eq:mod1}
\end{equation}
Here $A_{i,\alpha}$ is defined in $\eqref{E00}$ and studied in Lemma $\ref{L6}$.  The nonconstant parts of $A_{i,\alpha}$ correspond to the inhomogeneous term $\tilde{\ell}_{i,\alpha}$ which appears on the right hand side of the perturbation equation $\eqref{eq:perturb}$ and arise because the modulation parameters $c_\alpha$ and $\tau_\alpha$ are allowed to vary.  

The matrix $B_{\alpha}$, which is given by $$B_{\alpha} = \left( \ba{cc} 
\omega_\alpha(\partial_\tau \xi_{1,\alpha}, v_\alpha) 
& \omega_\alpha(\partial_c \xi_{1,\alpha},v_\alpha) \\ \\
\omega_\alpha(\partial_\tau \xi_{2,\alpha},v_{2,\alpha}) 
& \omega_\alpha(\partial_c \xi_{2,\alpha},v_{2,\alpha}) 
\ea \right)$$
corresponds to the inhomogeneous terms $\ell_{i,\alpha}$ which measure the failure of the range of the differential operator $L_\alpha$ to be symplectically orthogonal to the tangent vectors $\xi_{1,\alpha}$ and $\xi_{2,\alpha}$.  So long as the smallness condition $\eqref{eq:usmall}$ holds, it follows that 
$$A_{0,\alpha} - B_\alpha \sim \left( \ba{cc}
\eps^{5/2} (\|v_{1,\alpha}\|_{W_\alpha} + \|v_{2,\alpha}\|_\alpha) & \eps \\ 
\eps & \eps^{-2} \ea \right).$$
Here we have used $\eqref{eq:power_rules}$, $\eqref{eq:usmall}$, the fact that $\| \cdot \|$ bounds $\| \cdot \|_{W_\alpha}$, and the fact that $u_{c_\alpha}$ is exponentially localized.  Moreover, as before the off-diagonal terms are not only bounded above by a multiple of $\eps$ but also bounded below by a multiple of $\eps$.  Thus so long as $\eqref{eq:usmall}$ holds, $A_{0,\alpha} - B_\alpha$ is invertible.  Let $\hat{A}$ denote the matrix on the right hand side of $\eqref{eq:mod1}$.  In light of the fact that the off-diagonal terms of $\hat{A}$ are exponentially small, we see that $\hat{A}$ is invertible and to leading order in $\eps$ is given by 
$$\hat{A}^{-1} \sim \left( \ba{llll} \eps^{-4} & \eps^{-1} & {\tt exp} & {\tt exp} \\
\eps^{-1} & \eps^{1/2} (\|v_{1,+}\|_{W_+} + \|v_{2,+}\|_+)  & {\tt exp} & {\tt exp} \\ 
{\tt exp} & {\tt exp} & \eps^{-4} & \eps^{-1} \\ 
{\tt exp} & {\tt exp} & \eps^{-1} & \eps^{1/2} (\|v_{1,-}\|_{W_-} + \|v_{2,-}\|_-)  
\ea \right) $$
Thus we have proven
\begin{lemma} \label{lem:cdot_gammadot}
Assume that $\eqref{eq:usmall}$ holds.  Then there is a constant $K$ which does not depend on $\eps$ such that 
\begin{equation} \label{eq:cdot}
|\dot{c}_\alpha|  \le  K \left( \eps^{-1} | N_{1,\alpha} - \tilde{N}_{1,\alpha} | + \eps^{1/2} (\|v_{1,\alpha}\|_{W_\alpha} + \|v_{2,\alpha}\|_\alpha) | N_{2,\alpha} - \tilde{N}_{2,\alpha}|\right) 
 +  \left( \sum_{i = 1}^2 | N_{i,-\alpha} - \tilde{N}_{i,-\alpha} | \right){\tt exp}
\end{equation}
and
\begin{equation} \label{eq:gammadot}
|\dot{\gamma}_\alpha|  \le  K\left(
\eps^{-4} |N_{1,\alpha} - \tilde{N}_{1,\alpha}| + \eps^{-1}  |N_{2,\alpha} - \tilde{N}_{2,\alpha}| \right) + \left( \sum_{i=1}^2 |N_{i,-\alpha} - \tilde{N}_{i,-\alpha}| \right) {\tt exp} 
\end{equation}
\end{lemma}

We now refine Lemma $\ref{lem:cdot_gammadot}$ by obtaining quantitative estimates for the quantities which appear on the right hand side in $\eqref{eq:cdot}$ and $\eqref{eq:gammadot}$.
\begin{lemma}[Bounds on the Right Hand Side]
\label{lem:RHS_est_Mizu}
Let $g_\alpha$ be given by $\eqref{eq:gdef}$.
Assume that $\eqref{eq:usmall}$ holds.
Then the following hold
\begin{equation} \label{eq:JgX}
\| (Jg_\alpha)h_\alpha \|_\alpha \le K\eps^2 \|v_{1,\alpha}\|_{W_\alpha} + K(\|v_{1,\alpha}\| + \| v_{2,\alpha} \|)  \|v_{2,\alpha}\|_\alpha)  + {\tt exp} 
\end{equation}

\begin{equation} \label{eq:N1}
|N_{1,\alpha} - \tilde{N}_{1,\alpha}| \le K\eps^3 (\|v_{1,\alpha}\|_{W_\alpha}^2 + \|v_{2,\alpha}\|_\alpha^2 ) + {\tt exp} 
\end{equation}
and
\begin{equation} \label{eq:N2}
|N_{2,\alpha} - \tilde{N}_{2,\alpha}| \le K\left( \eps^{3/2} \|v_{1,\alpha}\|_{W_\alpha} + \eps^{-1/2}(\|v_{1,\alpha}\| + \| v_{2,\alpha} \|)  \|v_{2,\alpha}\|_\alpha) \right) {\tt exp}
\end{equation}
In particular, combining equation $\eqref{eq:cdot}$ and $\eqref{eq:gammadot}$ with equations $\eqref{eq:JgX}$, $\eqref{eq:N1}$, and $\eqref{eq:N2}$ 
we obtain
\begin{equation}
| \dot{c}_\alpha |  \le K (\eps^2 + \|v_{1,\alpha}\| + \|v_{2,\alpha}\|)(\|v_{1,\alpha}\|_{W_\alpha}^2 + \|v_{2,\alpha}\|_\alpha^2)  + {\tt exp}
\label{eq:cdot_RHS}
\end{equation}
and
\begin{equation}
 |\dot{\gamma}_\alpha | \le K \left(  \eps^{1/2} \|v_{1,\alpha}\|_{W_\alpha} + \eps^{-3/2}(\|v_{1,\alpha}\| + \|v_{2,\alpha}\|)\|v_{2,\alpha}\|_\alpha + \eps^{-1}(\|v_{1,\alpha}\|_{W_\alpha}^2 + \|v_{2,\alpha}\|_\alpha^2)\right) + {\tt exp}  
\label{eq:gammadot_RHS}
\end{equation}
 \end{lemma}
\begin{proof}

We first estimate $g_\alpha h_\alpha$.

Let $G(u,v) := H'(u + v) - H'(u) - H''(u)v$ so that 

\begin{equation} \ba{llll}
g_\alpha & = & G(u_{c_\alpha},v_\alpha) - G(0,v_\alpha) & (I) \\ \\
& & + G(v_{1\alpha},v_{2\alpha}) + (H''(v_{1\alpha})-1)v_{2\alpha} & (II) \\ \\
& & + (H''(u_{c_\alpha}) - 1)v_{1\alpha} & (III) \\ \\
& & + -H'(u_{c_{-\alpha}}) + H'(v_{1,\alpha}) - H'(v_1) + H'(u_{c_+} + u_{c_-} + v) - H'(u_{c_\alpha} + v_\alpha) & (IV) 
\ea
\end{equation}
Since $H$ is smooth, term $(I)$ is order $u_{c_\alpha} v_\alpha^2$.  Since $H''(0) = \mathrm{Id}$, term $(II)$ is order $|v_{2,\alpha}|(|v_{1,\alpha}| + |v_{2,\alpha}|)$ and term $(III)$ is order $v_{1,\alpha} u_{c_\alpha}$.  Term $(IV)$ becomes exponentially small in the weighted norm upon multiplying by $h_\alpha$.

We compute the terms $(I)$ - $(IV)$ in turn.
$$ \ba{lll}
\| I \|^2_\alpha & \le & K \sum_{k \in \Z} e^{2a(k - \tau_\alpha)} u_{c_\alpha}(k-\tau_\alpha)^2v_{\alpha}(k)^4 \\ \\
& \le & K\eps^4 \|v_\alpha\|_{\ell^\infty}^2 \sum_{k \in \Z} e^{-2(\kappa-a)|k - \tau_\alpha|} v_\alpha(k)^2 \\ \\
& \le & K\eps^4 \|v_\alpha\|_{\ell^\infty}^2 \left( \|v_{1,\alpha}\|_{W_\alpha}^2 + \|v_{2,\alpha}\|_\alpha^2 \right) \\ \\ \\
\| II \|_\alpha^2 & \le & K\sum_{k \in \Z} e^{-2\kappa |k - \tau_\alpha|} v_{2,\alpha}(k)^2(v_{1,\alpha}(k) + v_{2,\alpha}(k))^2 \\ \\
& \le & K(\|v_{1,\alpha}\| + \|v_{2,\alpha}\|)^2 \|v_{2,\alpha}\|^2_\alpha,  \\ \\
\| III \|_\alpha & \le & K\eps^2 \|v_{1,\alpha}\|_{W_\alpha} \\ \\
\| IV h_\alpha\|_\alpha  & = & {\tt exp}.
\ea
$$
In estimating $\| I \|_\alpha^2$ we have used the fact that $\kappa > 2a$.  In light of the fact that the $\alpha$ norm dominates the $W_\alpha$ norm, the contribution from $\| I \|_\alpha$ is dominated by the contributions from $\| II \|_\alpha$ and $\| III \|_\alpha$.  In particular
\begin{equation}
\| g_\alpha h_\alpha \|_\alpha \le K\left( \eps^2 \|v_{1,\alpha}\|_{W_\alpha} + (\|v_{1,\alpha}\| + \| v_{2,\alpha} \|)  \|v_{2,\alpha}\|_\alpha) + {\tt exp} \right). 
\label{eq:galpha_bound}
\end{equation}

Observe that 
\begin{equation}
(Jg_\alpha)h_\alpha = J(g_\alpha h_\alpha) + [J,h_\alpha]g_\alpha
\label{eq:bracket}
\end{equation}
where $[\cdot,\cdot]$ is the Lie bracket.  Thus
$\| (Jg_\alpha)h_\alpha \|_\alpha \le K \|g_\alpha h_\alpha\|_\alpha + \|[J,h_\alpha] g_\alpha\|_\alpha.$   Here we have used the fact that $J$ is a bounded operator.  We first examine the boundary terms.  For any vector $x$, the bracket $[J,h_\alpha]x$ is localized at the lattice sites $-1$ and $0$.  Thus in the weighted norm we have 
$$\|[J,h_\alpha]x\|_\alpha \le e^{-a \alpha \tau_\alpha} \|x\|_{\ell^\infty} \le e^{-a \alpha \tau_\alpha} \|x\|$$
and in particular $$\|[J,h_\alpha]g_\alpha\|_\alpha \le Ke^{-a \alpha \tau_\alpha} (1 + \|v\|) = {\tt exp}.$$
Here the $\|v\|$ term comes from terms $(I)$ - $(III)$ and the constant term comes from $(IV)$.  Thus $\eqref{eq:JgX}$ follows from $\eqref{eq:galpha_bound}$.

Moreover, in light of the facts that $\omega_\alpha(x,Jy) = - \langle y, x \rangle$ and $[J,h_\alpha]g_\alpha = {\tt exp}$ we can use the Cauchy-Schwartz inequality to obtain
$\tilde{N}_{2,\alpha} \le \eps^{-1/2} \|g_\alpha h_\alpha\|_\alpha + {\tt exp}$.  Similarly, we can use $\eqref{Perp1}$ and $\eqref{eq:power_rules}$ to see that $$N_{2,\alpha} = \omega_\alpha(\xi_{1,\alpha},v_{1,\alpha}) = \langle H'(u_{c_\alpha}), v_{1,\alpha} \rangle \le C\eps^{3/2} \|v_{1,\alpha}\|_{W_\alpha}.$$  Equation $\eqref{eq:N2}$ now follows from $\eqref{eq:galpha_bound}$.

We now estimate $N_{1,\alpha} - \tilde{N}_{1,\alpha} = - \omega_\alpha(\xi_{1,\alpha}, (Jg_\alpha)h_\alpha + L_\alpha v_{1,\alpha})$.
Define 
$$
\ba{llll}
\hat{g}_\alpha & := & 
H'(u_{c_+} + u_{c_-} + v) - H'(u_{c_+}) - H'(u_{c_-}) - H''(u_{c_\alpha})v \\ \\ 
& = & H'(u_{c_\alpha} + v_\alpha) - H'(u_{c_\alpha}) - H''(u_{c_\alpha})v_\alpha & (I) \\ \\
& & +  H'(u_{c_+} + u_{c_-} + v) - H'(u_{c_\alpha} + v_\alpha) - H'(u_{c_{-\alpha}}) - H''(u_{c_\alpha})v_{-\alpha}& (II)
\ea
$$
and
$$\tilde{g}_\alpha := H'(v_{1,\alpha}) - H''(u_{c_\alpha})v_{1,\alpha} \qquad (III)$$ so that $(Jg_\alpha)h_\alpha + L_\alpha v_{1,\alpha} = (J\hat{g}_\alpha)h_\alpha + (J\tilde{g})h_{-\alpha}$.  

Observe that $\| (II) h_\alpha\|_\alpha = {\tt exp}$ and $\tilde{g}h_{-\alpha} = 0$.  Since multiplication by $J$ and multiplication by $h_\alpha$ commute up to error terms at the lattice sites $0$ and $1$, the contribution of terms $(II)$ and $(III)$ to $\|(J\hat{g}_\alpha)h_\alpha + (J\tilde{g})h_{-\alpha}\|_\alpha$ is localized at the lattice sites $k \in \{-1,0,1\}$.  Thus this localized error term is exponentially small in the weighted norm.  Upon taking the symplectic inner product with $\xi_{1,\alpha}$ we see that the contribution of $(II)$ and $(III)$ to $N_{1,\alpha} - \tilde{N}_{1,\alpha}$ is ${\tt exp}$.

Since $H$ is smooth we see that $(I)(k,t) \le K v_\alpha(k,t)^2$ where $K$ may be chosen to be independent of $k$, $t$, and $\eps$ and to be increasing in $c_\alpha$ and $\| v_\alpha \|$, thus may be taken to be constant as we have assumed $c_\alpha$ and $\|v_\alpha\|$ to be bounded above.  In particular,
$$ \ba{lll}
\langle (I), \xi_{1,\alpha} \rangle & \le & K \eps^3 \sum_{k \in \Z} e^{-\kappa | k - \tau_\alpha | } v_\alpha(k)^2 \\ \\
& \le & K\eps^3 \sum_{k \in \Z} \left(e^{-a |k -\tau_\alpha|} v_{1,\alpha}(k) + e^{-a(k-\tau_\alpha)} v_{2,\alpha}(k)\right)^2 \\ \\ 
& \le & K\eps^3 (\| v_{1,\alpha} \|_{W_\alpha}^2 + \| v_{2,\alpha} \|_\alpha^2) \ea $$

In the second line we have used that $\kappa > 2a$.  In the last line we have used Young's inequality.  We may once again use the fact that $\omega_\alpha(x,Jy) = -\langle y, x \rangle$ together with the fact that $(J\hat{g}_\alpha)h_\alpha = J(\hat{g}_\alpha h_\alpha) + {\tt exp}$ to see that the contribution of term (I) to $|N_{1,\alpha} - \tilde{N}_{1,\alpha}|$ is order $\eps^3 (\|v_{1,\alpha}\|_{W_\alpha} + \|v_{2,\alpha}\|_\alpha)^2$.  
Combining these contributions we obtain $\eqref{eq:N1}$

This completes the proof.

\end{proof}

\section{Stability and the Proof of Theorem $\ref{thm:Orbital_Stability}$}

This section is concerned with controlling the evolution of $\|v(t)\|$.  
Lemma $\ref{Lem:Virial}$ controls $\|v_1\|_{W_\alpha}$, Lemma $\ref{apriori2}$ controls $\| v(t) \|$ in terms of $|c(t) - c(t_0)|$ and $\|v(t_0)\|$, and Theorem $\ref{T5}$ uses these bounds to control $|c(t) - c(t_0)|$ and thus $\| v(t) \|$.

We begin by controlling $\|v_1\|_{W_\alpha}$.  The following Lemma is a version of Lemma 9 in \cite{Mizumachi} where it was adapted from \cite{MM1}.  The only difference is that we pay attention to the small $\eps$ asymptotics of the constants and exponents.

\begin{lemma} \label{Lem:Virial}
Let $\bar{a} > 0$ be given.  Let $\theta : \R \to (0,1)$ be $C^2$ smooth and  assume that $A := \sup_{x \in \R} \sup_{0 < \delta < \bar{a}} \frac{\theta''(x+\delta)}{\theta'(x)} < \infty$.  
Then there are constants $C_0$, $C_1$, and $C_2$ which may be chosen independently of $\bar{a}$ such that 
whenever $\tau : \R \to \R$ satisfies $\dot{\tau}  > 1 + C_0 \bar{a}$, and $v$ is a solution of $\eqref{FPU}$ with $\| v(t_0) \| < C_1 \bar{a}$ then for any $t > t_0$ we have 
$$(C_2 / \bar{a}) \sum_{j \in \Z} \psi(j,t) |v(j,t)|^2 + \int_{t_0}^t \sum_{j \in \Z} \psi'(j,s) |v(j,s)|^2 dt \le (C_2 / \bar{a}) \sum_{j \in \Z} \psi(j,t_0) |v(j,t_0)|^2$$
where $\psi(x,t) = \theta(\bar{a} (x-\tau(t)))$, and $\psi'(x,t) = \bar{a} \theta'(\bar{a}(x-\tau(t)))$.
In particular, choosing $\theta(x) = 1 + \tanh(x)$, and $\bar{a} < < \eps^2$ we obtain 

\begin{equation}
2C\eps^{-4} \| \psi(t)^{1/2}v_1(t) \|^2 + \int_{t_0}^t \| v_1(s) \|_{W_\alpha(s)}^2 ds \le C \eps^{-4} \|v_1(t_0)\|^2
\label{eq:virial}
\end{equation}
\end{lemma}

\begin{rmk} The choice $\bar{a} < < \eps^2$ is necessary for the result to be valid when $\dot{\tau} - 1 \sim \eps^2$.
\end{rmk}

\begin{proof}
Sum by parts to obtain 
$$ \frac{d}{dt} \sum_{j \in \Z} \left( \frac{1}{2}p_j^2 + V(r_j)\right) \psi_j = 
-\sum_{j \in \Z} p_j V'(r_{j-1}) (\psi_j - \psi_{j-1}) - \dot{\tau}(\frac{1}{2}p_j^2 + V(r_j))\psi_j' .$$
Thus we may write the right hand side of the above as 
$$RHS = \sum_{j \in Z} \frac{-\dot{\tau}}{2}(p_j^2 + r_j^2) \psi_j' - p_j r_{j-1}(\psi_j - \psi_{j-1}) - p_j(V'(r_{j-1}) - r_{j-1})(\psi_j - \psi_{j-1}) +  (V(r_j) - \frac{1}{2}r_j^2)\psi_j' $$

In light of the convexity of $V$ and the conservation of the Hamiltonian, there is a constant $C$ which may be chosen uniformly  such that
 $$|V(r) - \frac{1}{2}r^2| + |V'(r) - r| \le C\|v(t_0)\| r^2 \le C \bar{a} \frac{1}{2} r^2 \qquad \mbox{ for } \qquad \|v(t_0)\| < C_1 \bar{a}$$ 
holds.  Here we are using $\|v(t)\|$ as a cheap bound for $r_j(t)$ and using the conservation and convexity of the Hamiltonian to replace $\|v(t)\|$ with $\|v(t_0)\|$.  We have also absorbed the constant $C_1$ into the constant $C$.  Now we have 
 $$ \ba{lll} RHS & \le & \frac{1}{2}\sum_{j \in \Z} -(\dot{\tau} - C\bar{a})(p_j^2 + r_j^2)\psi_j' + (1 + C\bar{a})|p_j r_{j-1}(\psi_j - \psi_{j-1}) \\ \\ $$
& \le & \frac{1}{2}\sum_{j \in \Z} -(\dot{\tau} - C\eps^2)(p_j^2 + r_j^2)\psi_j' + (1 + C\bar{a} )(\frac{\psi_j - \psi_{j-1}}{\psi_j'} \frac{1}{2}p_j^2\psi_j') + C\bar{a} )(\frac{\psi_{j+1} - \psi_j}{\psi_j'} \frac{1}{2}r_j^2\psi_j') \\ \\ 
& \le & -\frac{1}{2} \sum_{j \in Z} (\dot{\tau} - 1 - \bar{a}(1 + A + C(1 + \bar{a}A)) (p_j^2 + r_j^2)\psi_j' \\ \\
& \le & - C_2 \bar{a} \sum_{j \in \Z} (p_j^2 + r_j^2) \psi_j'
\ea $$
In the second line we have used Young's inequality and multiplied and divided by $\psi_j'$.  In the third line we have used the mean value theorem and the definition of $A$.  In the last line we have used the assumption $\dot{\tau} > 1  + C_0 \bar{a}$ and have chosen $C_0$ large enough.

Now integrate and use the fact that $V(r) > r$ for $r$ small to obtain the result.

\end{proof}

The following energy estimate is a consequence of the convexity of the Hamiltonian.

\begin{lemma} \label{apriori2}
Let $t_0$ and $t_1$ be real numbers.  Denote $m = \min\{t_0,t_1\}$ and $M = \max\{t_0,t_1\}$.  Suppose that $u = u_{c_+} + u_{c_-} + v$ solves $\eqref{FPU}$ for for $t \in [m,M]$.  Suppose also that there are wave speeds and phases $c(t_i)$ and $\tau(t_i)$ such that $\omega_\alpha(\xi_{i,\alpha}(\tau_\alpha(t_i),c_\alpha(t_i)),v_\alpha(t_i)) = 0$ for i = $0,1$.  

Then there are positive constants $\delta_1$ and $K$ such that whenever $\|v(t_i)\| < \delta_1$ for $i = 0,1$ then
\begin{equation}
\|v(t_1)\|^2 \le K \left[ \|v(t_0)\|^2 + \sum_{\alpha \in \{+,-\}} \eps |c_\alpha(t_1) - c_\alpha (t_0)| + 
e^{-\kappa \alpha \tau_\alpha(t_0)} (1 + \|v_\alpha(t_0)\|) + e^{-\kappa \alpha \tau_\alpha(t_1)} (1 + \|v_\alpha(t_1)\|) \right]
\label{apriori1}
\end{equation}
\end{lemma}

\begin{proof}
Since the Hamiltonian is convex in a neighborhood of $u_{c_+} + u_{c_-}$, there is a $\delta_1 > 0$, and constants $K_-$ and $K_+$ such that 
\begin{equation}
\frac{1}{2} K_- \|x\|^2 \le H(u_{c_+} + u_{c_-} + x) - H(u_{c_+} + u_{c_-}) - H'(u_{c_+} + u_{c_-}) x \le K_+ \|x\|^2
\label{eq:convex}
\end{equation}
whenever $\|x\| < \delta_1$.
Thus we may write
\begin{equation}
\ba{lllr}  \frac{1}{2}K_- \|v(t_1)\|^2 & \le & \left[H(u_{c_+} + u_{c_-} + v) - H(u_{c_+} + u_{c_-}) - <H'(u_{c_+} + u_{c_-}),v >\right|_{t_1} \\ \\
& = & \left[H(u_{c_+} + u_{c_-} + v) - H(u_{c_+} + u_{c_-}) - <H'(u_{c_+} + u_{c_-}),v>\right|_{t_0} & (I) \\ \\
& & + \left.H(u_{c_+} + u_{c_-}+v)\right|_{t_1} - \left.H(u_{c_+} + u_{c_-} + v)\right|_{t_0} & (II) \\ \\
& & + \left.H(u_{c_+} + u_{c_-})\right|_{t_0} - \left.H(u_{c_+} + u_{c_-})\right|_{t_1} & (III) \\ \\
& & \left.\left<H'(u_{c_+} + u_{c_-}),v\right>\right|_{t_0} - \left.\left<H'(u_{c_+} + u_{c_-}),v\right>\right|_{t_1} & (IV) \\ \\
\ea
\end{equation}

We immediately see that $\|(I)\| \le K\|v(t_0)\|^2$ from $\eqref{eq:convex}$ and the fact that $\|v\| < \delta_1$; we also see that $(II) = 0$ because the Hamiltonian is conserved along solutions.  To estimate $(III)$ we write 
$$ \ba{lllr}
H(u_{c_+} + u_{c_-}) & = & 
\sum_{\alpha \in \{+,-\}} (H(u_{c_+} + u_{c_-}) - H(u_{c_\alpha}))h_\alpha & (i) \\ \\
& & + \sum_{\alpha \in \{+,-\}} (H(u_{c_\alpha}) - H(\tilde{u}_\alpha))h_\alpha & (ii) \\ \\
& & + \sum_{\alpha \in \{+,-\}} H(\tilde{u}_\alpha) & (iii) \\ \\
& & + \sum_{\alpha \in \{+,-\}} H(\tilde{u}_\alpha)h_{-\alpha} &(iv)
\ea $$
where $\tilde{u}_\alpha(k,t) := u_{c_\alpha(t_0)}(k - \tau_\alpha(t))$.
It is a consequence of the mean value theorem that
$\| (i) \| \le K \sum_{\alpha \in \{+,-\}} \| u_{c_{-\alpha}} h_\alpha \| \le K (e^{-\kappa \tau_+} + e^{\kappa \tau_-})$.  From Lemma 9.1 in \cite{PF1} we know that $\frac{d H(u_c)}{dc} \sim \eps$, thus by the mean value theorem we see that $\| (ii) \|  \le K\eps \sum_{\alpha \in \{+,-\}} |c_\alpha(t) - c_\alpha(t_0)|$.  Since the Hamiltonian is translation invariant, we may replace  $(iii)$ with $H(u_{c_\alpha(t_0)}(k-c_\alpha(t_0)t))$ which is constant in time, thus does not contribute to $(III)$.  Term $(iv)$ is exponentially small because $u_\alpha h_\alpha$ is exponentially small and $H(0) = 0$.  Thus we obtain 
$$\| (III) \|  \le \sum_{\alpha \in \{+,-\} }e^{- \kappa \alpha \tau_\alpha(t_1)} + e^{- \kappa \alpha \tau_\alpha(t_0)} + \eps |c_\alpha(t_1) - c_\alpha(t_0)|.$$

To estimate $(IV)$ we write
$$\langle H'(u_{c_+} + u_{c_-}) , v \rangle = \sum_{\alpha \in \{+,-\} } \left\langle (H'(u_{c_+} + u_{c_-} ) - H'(u_{c_\alpha}) )h_\alpha, v_\alpha \right\rangle \le K ( e^{-\kappa \tau_+} \|v_+\| + e^{\kappa \tau_-}\|v_-\|).$$

Here we have used the mean value theorem and the fact that $\eqref{Perp1}$ may be written as $\langle H'(u_{c_\alpha}), v_\alpha \rangle = 0$.

Summing the bounds $I$ through $IV$ we obtain
$$
\|v(t_1)\|^2 \le K \left[ \|v(t_0)\|^2 + \sum_{\alpha \in \{+,-\}} \eps \;|c_\alpha(t_1) - c_\alpha (t_0)| + 
e^{-\kappa \alpha \tau_\alpha(t_0)} (1 + \|v_\alpha(t_0)\|) + e^{-\kappa \alpha \tau_\alpha(t_1)} (1 + \|v_\alpha(t_1)\|)
\right]
$$
as desired.
\end{proof}

We now study the semigroup associated with the linear variational problem about a solitary wave in the Toda lattice.  

\begin{prop} \label{prop:Toda_sg}
Let $V$ be the Toda potential $V(x) = e^{x} - 1- x$.  Fix $c_* > 1$.  Then there are positive constants $K$, $b$, and $a$ such that any solution $w$ of the linear equation
\begin{equation}
\partial_t w = JH''(u_c) w 
\label{eq:Toda_lin}
\end{equation}
 which also satisfies the orthogonality condition
\begin{equation}
 \langle J^{-1} \partial_x u_{c^*}, w \rangle = \langle J^{-1} \partial_c u_c |_{c = c^*}, w \rangle = 0
\label{eq:toda_perp}
\end{equation}
necessarily satsifies the decay estimate
\begin{equation}
\| e^{a(\cdot - ct)} w(t) \| \le Ke^{-b(t-t_0)} \|e^{a(\cdot - ct_0)} w(t_0)\|
\label{eq:decay_toda}
\end{equation}
Moreover the constant $K$ can be chosen uniformly in $c$ (and hence $\eps$) while the constants $a$ and $b$ satisfy $a = \mathcal{O}(\eps)$ and $b = \mathcal{O}(\eps^3)$ in the small $\eps$ regime for the scaling $\eqref{eq:KdV_approx}$.
\end{prop}
\begin{rmk} Here the inner product in $\eqref{eq:toda_perp}$ is the dual pairing between $\ell^2_{-a}$ and $\ell^2_a$.  The Neumann series form for $J^{-1}$ reflects this: $(I -S)^{-1}$ is the sum $\sum_{k=0}^\infty S^k$ which has norm $\frac{1}{1-e^{-a}}$ as an operator from $\ell^2_{-a}$ to itself.  This result is for rightward moving waves.  Of course, to examine leftward moving waves, one takes $\langle \cdot, \cdot \rangle$ to be the dual pairing between $\ell^2_a$ and $\ell^2_{-a}$ and defines $J^{-1} = -\sum_{k=1}^\infty S^{-k}$.  
\end{rmk}

The estimates in Proposition $\ref{prop:Toda_sg}$ are proven in \cite{MP}, with the exception of the fact that $K$ is independent of $c$, which turns out to be crucial for our proof of Theorem $\ref{thm:Orbital_Stability}$.  We establish this by studying the $\eps$-dependence of the B\"{a}cklund transformations constructed in \cite{MP}.  Details are provided in Appendix A.2.

We move now toward the main stability estimate.  Our aim is to use the exponential estimate $\eqref{eq:decay_toda}$ to solve the equation $\eqref{eq:perturb}$ via the variation of constants formula.  
Note that equation $\eqref{eq:toda_perp}$ may be rewritten in our notation as $\omega_+(\xi_i(c_+^*,\tau_+), w) = 0$ for $i = 1,2$.  
This differs from the orthogonality conditions $\eqref{Perp1}$ and $\eqref{Perp2}$ that we impose on $v_{2,\alpha}$ in three important ways.  First, $\xi_{i,\alpha}$ is evaluated at a fixed wavespeed $c_\alpha^*$ rather than at a modulated wavespeed $c_\alpha(t)$ and second, the condition $\omega_\alpha(\xi_{1,\alpha},v_{2,\alpha}) = 0$ is imposed rather than $\omega_\alpha(\xi_{1,\alpha},v_{2,\alpha}) = -\omega_\alpha(\xi_{1,\alpha},v_{1,\alpha})$.  Third, we consider any potential $V$ which satisfies $(H1)$ which is more general than restricting to the Toda potential.  More concretely, we have only proven that the estimate $\eqref{eq:decay_toda}$ holds for perturbations which are symplectically orthogonal to the tangent vectors of the wave state manifold {\it for the Toda model} and at a {\it fixed reference wave speed} $c^*_\alpha$.  Our perturbation $v_{2,\alpha}$ satisfies different conditions, $\eqref{Perp1}$ and $\eqref{Perp2}$, which impose symplectic orthogonality to the tangent vectors of the wave state manifold {\it for the FPU model} and at a {\it modulated wave speed} $c_\alpha(t)$.  

Furthermore in Proposition $\ref{prop:Toda_sg}$ we have studied the semigroup for the Toda lattice linearized about a solitary wave while the linear part of our evolution equation generates a different semigroup - that of a more general FPU model linearized about a solitary wave.  Thus the solitary wave profile $u_c$ which appears both in the orthogonality condition $\eqref{eq:toda_perp}$ and the evolution equation $\eqref{eq:Toda_lin}$ will differ between this proposition and our later analysis.  

The reason for these differences are that $\eqref{Perp1}$ and $\eqref{Perp2}$ are essential for obtaining good bounds on the modulation equation while $\eqref{eq:toda_perp}$ is the most natural condition (and the one used in \cite{PF4} and subsequently in \cite{MP}) for extending the KdV semigroup estimates of \cite{PW} to the FPU context.  Moreover, it is useful to work with the Toda potential because the B\"{a}cklund transformation constructed in \cite{MP} allows us to obtain uniform estimates in the small $\eps$ regime for the constant in $\eqref{eq:decay_toda}$.  These differences contribute additional terms to the evolution equation for $v(t)$ which we bound and regard as small inhomogeneous terms on the right hand side of $\eqref{eq:perturb}$.  This is addressed in Lemma $\ref{lem:Qalpha}$.

We now define a map which takes as an input $v_{2,\alpha}$ satisfying $\eqref{Perp1}$ and $\eqref{Perp2}$ and gives as its output a $\tilde{v}_{2,\alpha}$ satisfying $\eqref{eq:toda_perp}$. 
Given an initial time $t_0$, define $\tilde{\xi}_{1,\alpha}(\tau) := \partial_x u_{c_\alpha(t_0)}^T(\cdot - \tau)$ and $\tilde{\xi}_{2,\alpha}(\tau) := \partial_c u_{c_\alpha(t_0)}^T(\cdot - \tau)$ where the superscript $T$ denotes that these are solitary waves for the Toda lattice.  Given $v_{2,\alpha}$ satisfying $\eqref{Perp1}$ and $\eqref{Perp2}$ we seek $\tilde{v}_{2,\alpha}$ close to $v_{2,\alpha}$ which satisfies the conditions $\eqref{eq:toda_perp}$ which we write as 
$$\omega_\alpha(\tilde{\xi}_{i,\alpha}, \tilde{v}_{2,\alpha} ) = 0 \qquad i = 1,2$$
so that we may apply the semigroup estimates $\eqref{eq:decay_toda}$ of Proposition $\ref{prop:Toda_sg}$ to control the evolution of $\tilde{v}_{2,\alpha}$.

To that end, we define a projection $Q_\alpha$ onto the symplectically orthogonal space as follows:

\begin{equation} \label{eq:Qdef}
Q_\alpha v := v - \frac{\omega_\alpha(\tilde{\xi}_{2,\alpha}, v)}{\omega_\alpha(\tilde{\xi}_{2,\alpha},\tilde{\xi}_{1,\alpha})}\tilde{\xi}_{1,\alpha} - \frac{\omega_\alpha(\tilde{\xi}_{2,\alpha}, \tilde{\xi}_{2,\alpha})\omega_\alpha(\tilde{\xi}_{1,\alpha},v)}{\omega_\alpha (\tilde{\xi}_{2,\alpha},\tilde{\xi}_{1,\alpha})^2} \tilde{\xi}_{1,\alpha} - \frac{\omega_\alpha(\tilde{\xi}_{1,\alpha},v)}{\omega_\alpha(\tilde{\xi}_{1,\alpha},\tilde{\xi}_{2,\alpha})} \tilde{\xi}_{2,\alpha}
\end{equation}
We remark that the denominator $\omega_\alpha(\tilde{\xi}_{1,\alpha},\tilde{\xi}_{2,\alpha})$ is bounded below by a multiple of $\eps$ as was shown in Lemma 9.1 of \cite{PF1} and Lemma 2.1 in \cite{PF2}.

Define 
$$\tilde{L}_\alpha = \partial_t - \frac{\dot{\tau}_\alpha}{c_\alpha(t_0)}JH_{Toda}''(u^T_{c_\alpha(t_0)})$$

\begin{lemma} \label{lem:Qalpha}
Suppose that $v_{2,\alpha}$ satisfies $\eqref{Perp1}$ and $\eqref{Perp2}$ and $|c(t) - c(t_0)| < C\eps^3$.  Then for $\eps$ sufficiently small, 
\begin{equation}
\| v_{2,\alpha} \|_\alpha \le C(\| Q_\alpha v_{2,\alpha}\|_\alpha + \|v_{1,\alpha}\|_{W_\alpha}) 
\label{eq:Q2}
\end{equation}
Moreover
$$[ \tilde{L}_\alpha, Q_\alpha] = 0$$
Furthermore, there is a $C$ which is independent of $\eps$ such that $\|Q_\alpha x\|_\alpha \le C\|x\|_\alpha$
\end{lemma}
\begin{proof}
To see that the norm of $Q_\alpha$ is bounded uniformly in $\eps$, apply $\eqref{eq:power_rules}$ to $\eqref{eq:Qdef}$.
We now establish $\eqref{eq:Q2}$.  We estimate
$$ \ba{lllr}
\omega_\alpha(\tilde{\xi}_{2,\alpha},\tilde{\xi}_{1,\alpha}) (I - Q_\alpha) v_{2,\alpha} & = &
\omega_\alpha(\tilde{\xi}_{2,\alpha},v_{2,\alpha})\tilde{\xi}_{1,\alpha} & (i) \\ \\ 
& & + \omega_\alpha(\tilde{\xi}_{1,\alpha},v_{2,\alpha})
\left( \frac{\omega_\alpha(\tilde{\xi}_{2,\alpha},\tilde{\xi}_{2,\alpha})  \tilde{\xi}_{1,\alpha}}{\omega_\alpha(\tilde{x}_{1,\alpha},\tilde{\xi}_{2,\alpha})} + \tilde{\xi}_{2,\alpha}\right)  & (ii)
\ea 
$$
Observe that since the wave profiles $u_c$ and $u_c^T$ are both equal to the KdV wave profile to leading order in $\eps$, their difference and the derivatives of their difference is higher order in $\eps$.  To quantify this use $\eqref{eq:eps^2}$ and the triangle inequality to obtain
\begin{equation}
\| \partial_c^k \partial_\tau^j (u_c - u_c^T)\|_\alpha \le C \eps^{7/2 + j - 2k},
\label{eq:power_rules2}
\end{equation}
which is reminscent $\eqref{eq:power_rules}$.  From the above, it follows that 
$$ \| \tilde{\xi}_{2,\alpha} - \xi_{2,\alpha} \|_\alpha \le 
\| \partial_c u_{c(t)} - \partial_c u_{c(t_0)} \|_\alpha + \| \partial_c u_{c(t_0)} - \partial_c u_{c(t_0)}^T \|_\alpha \le  C\left( \eps^{-5/2} |c(t) - c(t_0)| + \eps^{3/2} \right).$$
here we have used $\eqref{eq:power_rules}$ and the Mean Value Theorem to bound the first term and $\eqref{eq:power_rules2}$ to bound the second term.  Similarly, we obtain
$$ \| \tilde{\xi}_{1,\alpha} - \xi_{1,\alpha} \|_\alpha \le 
\| \partial_\tau u_{c(t)} - \partial_\tau u_{c(t_0)} \|_\alpha + \| \partial_\tau u_{c(t_0)} - \partial_\tau u_{c(t_0)}^T \|_\alpha \le C \left( \eps^{1/2} |c(t) - c(t_0)| + \eps^{9/2}\right).$$
Use $\eqref{eq:power_rules}$ together with the above and
$$\omega_\alpha(\tilde{\xi}_{1,\alpha},v_{2,\alpha}) = \left[\omega_\alpha(\tilde{\xi}_{1,\alpha} - \xi_{1,\alpha},v_{2,\alpha}) + \omega(\xi_{1,\alpha},v_{1,\alpha})\right]$$
and
$$\omega_\alpha(\tilde{\xi}_{2,\alpha},v_{2,\alpha}) = \omega_\alpha(\tilde{\xi}_{2,\alpha} - \xi_{2,\alpha},v_{2,\alpha})$$
which follow from the orthogonality conditions $\eqref{Perp1}$ and $\eqref{Perp2}$ to obtain

$$(i) \le (\eps^{-3/2} |c(t) - c(t_0)| + \eps^{5/2})\|v_{2,\alpha}\|_\alpha$$ and 
$$(ii) \le C \eps \|v_{1,\alpha}\|_{W_\alpha} + C(\eps^{-1} |c(t) - c(t_0)| + \eps^3)\|v_{2,\alpha}\|_\alpha
$$

Summing these estimates and dividing by $\omega_\alpha(\tilde{\xi}_{2,\alpha},\tilde{\xi}_{1,\alpha}) \sim \eps$, (which is known from $P3$), we obtain
$$\| (I - Q_\alpha) v_{2,\alpha} \| \le C\| v_{1,\alpha} \|_{W_\alpha} + C\left( \eps^{-5/2} | c(t_0) - c(t) | + \eps^{3/2} \right) \|v_{2,\alpha} \|_\alpha $$
After applying the triangle inequality and using the fact that $\eps^{-5/2} | c(t_0) - c(t) | < < 1$ we obtain
$\eqref{eq:Q2}$ as desired.

We now show that $Q_\alpha$ commutes with $\tilde{L}_\alpha$.  If the following computation $H := H_{Toda}$.
Write $Q_\alpha v = v - a_1 \omega_\alpha (\tilde{\xi}_{2,\alpha},v)\tilde{\xi}_{1,\alpha} - a_2\omega_\alpha(\tilde{\xi}_{1,\alpha},v)\tilde{\xi}_{1,\alpha} + a_1\omega(\tilde{\xi}_{1,\alpha},v)\tilde{\xi}_{2,\alpha}$ where
$a_1 = \frac{1}{\omega_\alpha(\tilde{\xi}_{2,\alpha},\tilde{\xi}_{1,\alpha})}$ and $a_2 = \frac{\omega_\alpha(\tilde{\xi}_{2,\alpha},\tilde{\xi}_{2,\alpha})}{\omega(\tilde{\xi}_{1,\alpha},\tilde{\xi}_{2,\alpha})}$.  Note that $a_1$ and $a_2$ are constant in $t$ because $\tilde{\xi}_{i,\alpha}$ are evaluated at a fixed wave speed $c_\alpha^*$.

Also recall that $\partial_t \tilde{\xi}_{1,\alpha} = \frac{\dot{\tau}}{c_0} JH''(u_{c_0})\tilde{\xi}_{1,\alpha}$ and $\partial_t \tilde{\xi}_{2,\alpha} = \frac{\dot{\tau}}{c_0} JH''(u_{c_0})\tilde{\xi}_{2,\alpha} - \frac{1}{c_0}\tilde{\xi}_{1,\alpha}$.

One may compute
$$\ba{lll} 
\frac{d}{dt} Qv & = & \partial_t v - a_1 \omega_\alpha(\tilde{\xi}_{2,\alpha},\tilde{L}_\alpha v)\tilde{\xi}_{1,\alpha} + \frac{a_1}{c_0}\omega_\alpha(\tilde{\xi}_{1,\alpha},v)\tilde{\xi}_{1,\alpha} - a_2\omega_\alpha(\tilde{\xi}_{1,\alpha},\tilde{L}_\alpha v)\tilde{\xi}_{1,\alpha} + a_1\omega_\alpha(\tilde{\xi}_{1,\alpha},\tilde{L}_\alpha v)\tilde{\xi}_{2,\alpha} \\ \\
& & - \frac{\dot{\tau}}{c_0} JH''(u_{c_0}) \left( v-Qv \right) - \frac{a_1}{c_0}\omega_\alpha(\tilde{\xi}_{1,\alpha},v)\tilde{\xi}_{1,\alpha} \\ \\
& = & \frac{\dot{\tau}}{c_0} JH'''(u_{c_\alpha})v + \tilde{L}_\alpha v - a_1 \omega_\alpha(\tilde{\xi}_{2,\alpha},\tilde{L}_\alpha v)\tilde{\xi}_{1,\alpha}  - a_2\omega_\alpha(\tilde{\xi}_{1,\alpha},\tilde{L}_\alpha v)\tilde{\xi}_{1,\alpha} + a_1\omega_\alpha(\tilde{\xi}_{1,\alpha},\tilde{L}_\alpha v)\tilde{\xi}_{2,\alpha}
\ea
$$
Thus
$\tilde{L}_\alpha Q_\alpha v = Q_\alpha \tilde{L}_\alpha v$ as desired.
\end{proof}

Thus
\begin{equation}
 \ba{lll}
\tilde{L}_\alpha Q_\alpha v_{2,\alpha} & = & Q_\alpha L_\alpha v_{2,\alpha} + Q_\alpha (\tilde{L}_\alpha - L_\alpha)v_{2,\alpha} \\ \\
& = & Q_\alpha \left( (Jg_\alpha)h_\alpha + \tilde{\ell}_\alpha + (\tilde{L}_\alpha - L_\alpha)v_{2,\alpha}\right)   \\ \\
& := &  G_\alpha
\ea
\label{eq:stability_perturb}
\end{equation}

Let $\Psi_\alpha$ be the evolution semigroup associated to equation $\eqref{eq:Toda_lin}$.  Then solutions of $\eqref{eq:stability_perturb}$ have the following representation via the variation of constants formula.
\begin{equation}
Q_\alpha v_{2,\alpha}(t) =  \Psi_\alpha(\tau_\alpha(t),\tau_\alpha(t_0))Q_\alpha v_{2,\alpha}(t_0) +  \int_{t_0}^t \Psi_\alpha(\tau_\alpha(t),\tau_\alpha(s))G_\alpha(s) ds \label{E11}
\end{equation}

In light of the exponential estimates on $\Psi_\alpha$ given by $\eqref{eq:decay_toda}$ we are now in a position to control $\|v_{2,\alpha}\|_\alpha$.  

\begin{thm}[Stability Estimates] \label{T5}
Assume (H1).  For each $\eta > 0$, there exist positive constants $C_*$ and $\delta_*$ which may be chosen independent of $\eps$, with the following property: 

Suppose that $u$ is a solution of $\eqref{FPU}$ and
$$u(t) = u_{c_+(t)}(\cdot - \tau_+(t)) + u_{c_-(t)}(\cdot - \tau_-(t)) + v_1(t) + v_2(t) \qquad t \in [t_0,t_1]$$
Suppose further that the solutions satisfy
$$|c(t) - c(t_0)| + \|v_2(t)\| + \|v_1(t)\| + \|v_{2,+}(t)\|_+ + \|v_{2,-}(t)\|_- < C\eps^{3 + \eta} \qquad \mbox{ for all } t \in [t_0,t_1].$$
Then for all $t \in [t_0,t_1]$ the following inequalities necessarily hold: 
\begin{equation}
|c_\alpha(t) - c_\alpha(t_0)| \le C\left[ \eps^{-4} \|v_1(t_0)\|^2 + \eps^{-1} \|v_{2,\alpha}(t_0)\|_\alpha^2 \right] + {\tt exp}(t_0) 
\label{eq:T5E1}
\end{equation}
and
\begin{equation}
\|v(t)\|^2 \le C\left[ \eps^{-3}\|v_1(t_0)\|^2 + \|v_2(t_0)\|^2 + \|v_{2,+}(t_0)\|_+^2 + \|v_{2,-}(t_0)\|_-^2 \right] + {\tt exp}(t_0)
\label{eq:T5E2}
\end{equation}
and
\begin{equation}
\|v_{2,\alpha}(t) \|_\alpha^2 \le C\left[\eps^{-3} \|v_1(t_0)\|^2 + \|v_{1,\alpha}(t)\|_{W_\alpha(t)}^2 + e^{-2b(t-t_0)}\|v_{2,\alpha}(t_0)\|_\alpha^2 +e^{-2bt}  \right].
\label{eq:T5E3}
\end{equation} 
\end{thm}

\begin{proof}
We use the variation of constants formula $\eqref{E11}$ to solve $\eqref{eq:stability_perturb}$.  From $\eqref{eq:gammadot_RHS}$ and the smallness assumptions in the statement of the theorem it follows that $|\dot{\gamma}| \le C\eps^{7/2 + \eta}$ for $t \in [t_0,t_1]$.  This together with the smallness assumption on the variation in $c$ yields $\dot{\tau}_\alpha(t) > c_\alpha(t_0) + c_\alpha(t) - c_\alpha(t_0) + \dot{\gamma}_\alpha(t) > 1 + C\bar{a}$.  Thus $\tau_\alpha(t) - \tau_\alpha(s) \ge t -s$ whenever $t_0 < s < t < t_1$, and in addition the use of $\eqref{eq:virial}$ is justified.

In light of the decay estimate $\eqref{eq:decay_toda}$ we have the estimate
 \begin{equation} 
\|Q_\alpha v_{2,\alpha}(t)\|_\alpha  \le Ce^{-\alpha b(t - t_0)} \| Q_\alpha v_{2,\alpha}(t_0)\|_\alpha  + C \int_{t_0}^t e^{- \alpha b (t-s)} \|G_\alpha(s)\|_\alpha ds. \label{eq:varconst}
\end{equation}
For ease of notation we have replaced $\tau_\alpha(t) - \tau_\alpha(s)$ with $t-s$ in the exponents.  This is justified by the remark just prior to $\eqref{eq:varconst}$.
Use $\eqref{eq:Q2}$ as well as the fact that the norm of $Q_\alpha$ is bounded uniformly in $\eps$ to obtain
\begin{equation}
\|v_{2,\alpha}(t) \|_\alpha \le Ce^{-\alpha b(t - t_0)}\|v_{2,\alpha}(t_0)\|_\alpha +  C \|v_{1,\alpha}(t) \|_{W_\alpha} + C \int_{t_0}^t e^{- \alpha b (t - s)} \|G_\alpha(s)\|_\alpha ds. \label{eq:varconst2}
\end{equation}

In order to proceed we make the estimate
\begin{equation} \label{eq:Gest}
\ba{lll}
\|G_\alpha(s)\|_\alpha 
& \le & C\|(Jg_\alpha(s))h_\alpha\|_\alpha + C \|\tilde{\ell}_\alpha(s)\|_\alpha + C( |c_\alpha(s)-c_\alpha(t_0)| + |\dot{\gamma_\alpha}(s)| + \eps^4) \|v_{2,\alpha}\|_\alpha \\ \\
& \le & C \eps^2 \|v_{1,\alpha} \|_{W_\alpha} +C (\|v_{1,\alpha}\| + \|v_{2,\alpha}\|)\|v_{2,\alpha}\|_\alpha + C\eps^{-1/2} |\dot{c}_\alpha | + \eps^{5/2} |\dot{\gamma}_\alpha| \\ \\
& & \qquad + \; C(|\dot{\gamma}_\alpha| + |c_\alpha(s) - c_\alpha(t_0)| + \eps^4) \|v_{2,\alpha}\|_\alpha + {\tt exp} \\ \\
& \le & C\eps^2 \|v_{1,\alpha}\|_{W_\alpha} + C\eps \|v_{1,\alpha}\|_{W_\alpha}^2 + C\eps \|v_{2,\alpha}\|_\alpha^2 \\ \\ 
& & \qquad + C (\|v_{1,\alpha}\| + \|v_{2,\alpha}\| + |c_\alpha(s) - c_\alpha(t_0)| + \eps^4) \|v_{2,\alpha}\|_\alpha  + {\tt exp} \\ \\
& \le & C \eps^2 \|v_{1,\alpha}\|_{W_\alpha} + C \eps^{3 + \eta}\| v_{2,\alpha}\|_\alpha + {\tt exp} \qquad t \in [t_0,t_1] \ea
\end{equation}

In the first line we have estimated $Q_\alpha( \tilde{L}_\alpha - L_\alpha)v_{2,\alpha}$ in $\eqref{eq:stability_perturb}$ and also used the fact that $\|Q_\alpha\|$ is bounded uniformly in $\eps$.  In the second line we have used equations $\eqref{eq:JgX}$ and $\eqref{eq:ldef}$. In the third line we have used equations $\eqref{eq:cdot_RHS}$ and $\eqref{eq:gammadot_RHS}$.  In the fourth line we have used the smallness assumptions in the statement of the theorem.

Substitute $\eqref{eq:Gest}$ into $\eqref{eq:varconst2}$ to obtain
$$ \ba{lll}
\|  v_{2,\alpha} (t) \|_\alpha & \le & Ce^{-b(t-t_0)} \|v_{2,\alpha}(t_0)\|_\alpha +  C \| v_{1,\alpha}(t) \|_{W_\alpha} C\eps^2 \int_{t_0}^t e^{-b(t-s)} \|v_{1,\alpha}(s)\|_{W_\alpha(s)} ds \\ \\ 
& & \qquad +   \eps^{3 + \eta} \int_{t_0}^t e^{-b(t-s)} \|v_{2,\alpha}(s)\|_\alpha ds + {\tt exp}
\ea $$

Apply the integral form of Gronwall's inequality to obtain
\begin{equation}
\ba{lll}
 \|v_{2,\alpha}(t)\|_\alpha & \le & Ce^{C (\eps^{3 + \eta})(1 - e^{-b(t-t_0)})/b}\left[\|v_{1,\alpha}(t)\|_{W_\alpha(t)} + e^{-b(t-t_0)} \|v_{2,\alpha}(t_0)\|_\alpha \right. \\ \\
 & & \qquad \; \left. + \eps^2\int_{t_0}^t e^{-b(t-s)} \|v_{1,\alpha}(s)\|_{W_\alpha} ds\right]  + e^{-bt}{\tt exp(t_0)} \\ \\
 & \le & C\left[ \|v_{1,\alpha}(t)\|_{W_\alpha(t)} + e^{-b(t-t_0)}\|v_{2,\alpha}(t_0)\|_\alpha \right. \\ \\
 & & \qquad \left. + \eps^2 \int_{t_0}^t e^{-b(t-s)} \|v_{1,\alpha}(s)\|_{W_\alpha} ds \right] + e^{-bt}{\tt exp}(t_0) \ea 
\label{eq:Gronwall1}
\end{equation}
Here we have used the fact that $b = \mathcal{O}(\eps^3) > C\eps^{3 + \eta}$ for $\eps$ small enough.  Use Young's inequality to obtain
\begin{equation}
\ba{lll} \|v_{2,\alpha}(t)\|_\alpha^2 & \le &
C\left[  \|v_{1,\alpha}(t)\|_{W_\alpha(t)}^2 + e^{-2b(t-t_0)}\|v_{2,\alpha}(t_0)\|_\alpha^2 \right. \\ \\
& & \qquad \left. + \eps^4 \left(  \int_{t_0}^t e^{-b(t-s)} \|v_{1,\alpha}(s)\|_{W_\alpha(s)} ds \right)^2 \right]+ e^{-2bt} {\tt exp}(t_0) 
\ea
\label{eq:young}
\end{equation}

Now compute
$$ \ba{lll}
\int_{t_0}^{t_1} \left( \int_{t_0} e^{-b(t-s)} \|v_{1,\alpha}(s)\|_{W_\alpha(s)} ds \right)^2 dt & 
\le & \int_{t_0}^{t_1} \left( \int_{t_0}^t e^{-b(t-s)}ds \int_{t_0}^t e^{-b(t-s)} \|v_{1,\alpha}(s)\|_{W_\alpha(s)}^2 ds \right) dt \\ \\
& \le & b^{-1} \int_{t_0}^{t_1} \|v_{1,\alpha}(s)\|_{W_\alpha(s)}^2 \int_s^{t_1} e^{-b(t-s)} dt ds\\ \\ 
& \le & b^{-2} \int_{t_0}^{t_1} \|v_{1,\alpha}(s)\|_{W_\alpha(s)}^2 ds
\ea $$

In the second line we have used the Cauchy-Schwartz inequality and changed the order of integration.  Recall that $b = \mathcal{O}(\eps^3)$ and substitute the above into $\eqref{eq:young}$ to obtain
\begin{equation}
 \int_{t_0}^{t_1} \| v_{2,\alpha} (t) \|_\alpha^2 dt \le C \left[ \eps^{-2} \int_{t_0}^{t_1} \| v_{1,\alpha}(t) \|_{W_\alpha(t)}^2 dt + \eps^{-3}\|v_{2,\alpha}(t_0)\|_\alpha^2 + e^{-bt}{\tt exp}(t_0)\right]
\label{eq:L2_bound}
\end{equation}
Apply Lemma $\ref{Lem:Virial}$, and in particular $\eqref{eq:virial}$ to see that 
\begin{equation}
\int_{t_0}^{t_1} \|v_{2,\alpha} (t)\|_\alpha^2 dt  \le C\left[  \eps^{-6} \|v_1(t_0)\|_{W_\alpha(t_0)}^2 + \eps^{-3} \|v_{2,\alpha}(t_0)\|_\alpha^2 + {\tt exp}(t_0)\right]. \label{eq:L2_2}
\end{equation}
Integrate equation $\eqref{eq:cdot_RHS}$ and use $\eqref{eq:L2_2}$ and $\eqref{eq:virial}$ to obtain $\eqref{eq:T5E1}$.  Substitute $\eqref{eq:T5E1}$ into equation $\eqref{apriori1}$ to obtain $\eqref{eq:T5E2}$.
To obtain $\eqref{eq:T5E3}$, apply the Cauchy-Schwartz inequality to the integral term in $\eqref{eq:young}$ and use $\eqref{eq:virial}$.  
This completes the proof. 
\end{proof}

\begin{rmk}
It might appear from the proof that the cross terms associated with $u_{c_{-\alpha}}$ do not appear in the estimates for $v_\alpha$.  All of these cross terms are exponentially small, and incorporated into the term {\tt exp}.
\end{rmk}

Theorem $\ref{thm:Orbital_Stability}$ may now be regarded as a corollary of Theorem $\ref{T5}$.

\begin{proof}[Proof of Theorem $\ref{thm:Orbital_Stability}$]
Proposition $\ref{Ptube}$ applies as long as $\|v_{2,+}\|_+ +\|v_{2,-}\|_- + \|v_1\| < C\eps^{5/2 + \eta}$.  Since this inequality is satisfied by hypothesis at $t = t_0$ and we have continuous dependence on initial conditions, it follows that tubular coordinates remain valid on some interval $[t_0,t_1]$ with $t_1 > t_0$.  Thus we may apply Theorem $\ref{T5}$ to obtain $\eqref{eq:T5E1}$,$\eqref{eq:T5E2}$, and $\eqref{eq:T5E3}$ for $t \in [t_0,t_1]$.  It is straightforward to see that so long as $\|v_1(t_0)\| < C\eps^{9/2 + \eta}$ and $\|v_{2,\alpha}(t_0)\|_\alpha < C\eps^{3 + \eta}$ with $0 < \eta < 1/2$ then we may take $t_1 = \infty$.  In particular $\eqref{eq:T5E2}$ holds for all $t > t_0$.  This establishes $\eqref{eq:Thm2_1}$.  That $c_\alpha$ converges follows from the fact that $\dot{c}_\alpha$ is integrable.  

From $\eqref{eq:gammadot_RHS}$ we see that to show that $\dot{\gamma} \to 0$, it suffices to show that $\|v_{1,\alpha}\|_{W_\alpha}$ and $\|v_{2,\alpha}\|_\alpha$ each go to zero.
To see that $\| v_{1,\alpha}(t)\|_{W_\alpha(t)} \to 0$ we compute
$$ \frac{d}{dt} \| v_{1,\alpha}(t)\|^2_{W_\alpha(t)} \le \| v_1(t)\| \| JH'(v_1(t)) \|_{W_\alpha} + |\dot{\tau}(t)| \| v_1(t) \|_{W_\alpha(t)} \le C$$
This, together with $\eqref{eq:virial}$ and in particular the fact that $\| v_{1,\alpha}(t)\|_{W_\alpha}$ is square-integrable imply that $\| v_{1,\alpha}(t) \|_{W_\alpha(t)} \to 0$ as desired.  Similarly, 
\[ \frac{d}{dt} \|v_{2,\alpha}\|_\alpha^2 \le 2 a \alpha \dot{\tau}_\alpha \|v_{2,\alpha}\|_\alpha^2 + 2 \|v_{2,\alpha}\|_\alpha \| \partial_t v_{2,\alpha}\|_\alpha. \]
Since $v_{2,\alpha}$ solves $\eqref{eq:perturb}$ and $\eqref{eq:T5E3}$ holds, all of the terms on the right are bounded; this together with the fact that $\|v_{2,\alpha}\|_\alpha$ is square-integrable implies that $\|v_{2,\alpha}(t)\|_\alpha \to 0$ as $t \to \infty$ as desired.
It remains only to establish $\eqref{eq:thm2_2}$ and $\eqref{eq:thm2_3}$, both of which follow from $\eqref{eq:T5E3}$.
This completes the proof.
\end{proof}

\section*{Acknowledgements}
The authors thank T. Mizumachi and R.L. Pego for helpful discussions, in particular regarding their unpublished work \cite{MP}.  They also thank T. Mizumachi for suggesting the method of proof used in Proposition $\ref{prop:Toda_sg}$.  This work was funded in part by the National Science Foundation under grants DMS-0603589 and DMS-0405724.

\appendix
\section{Uniform Bounds in the Small $\eps$ Regime}

\subsection{Uniform Bounds in the Small $\eps$ Regime for Higher Derivatives of the Wave Profile}

\begin{proof}[Proof of Lemma $\ref{lem:small_eps_1}$]

First we observe some facts about the weighted space $H^s_a$ with norm $\| f(\cdot)  \|_{H^s_a} = \| f(\cdot) \|_{H^s} + \| e^{a \cdot} f(\cdot) \|_{H^s}$.  For $L \in \mathcal{L}(H^s_a)$ we estimate
$\|L \psi\|_{H^s_a} = \|L \psi \|_{H^s} + \| L_a \psi_a \|_{H^s}$ where $L_a := e^{a \cdot} L e^{-a \cdot}$ and $\psi_a(\cdot) = e^{a \cdot} \psi(\cdot)$.  Thus 
\begin{equation}
\|L\|_{\mathcal{L}(H^s_a)} \le \| L \|_{\mathcal{L}(H^s)} + \|L_a\|_{\mathcal{L}(H^s)}.
\label{eq:H1a}
\end{equation}
Also observe that whenever $P$ is a pseudodifferential operator with symbol $p$, then $\| P \|_{\mathcal{L}(H^s)} \le \| p \|_{L^\infty(\R,\mathbb{C})}$ for any $s \ge 0$.  Furthermore, if we denote the symbol of $P_a := e^{a \cdot} P e^{-a \cdot}$ by $p_a$ then $p_a(\xi) = p(\xi + ia)$.  Thus to bound the operator norm of a pseudodifferential operator (as in $(1)$) in $H^s_a$ it suffices to bound its symbol uniformly in a strip $|\mathrm{Im} \xi| \le a$ in the complex plane.  A slight modification of Lemma 3.3 in \cite{PF1} yields this bound.  

More explicitly, let $p(\xi) = \frac{\sinc^2(\xi/2)}{c^2 - \sinc^2(\xi/2)}$ with $c^2 = 1 + \eps^2 \beta / 12$ 
Denote by $p_*$ the {\it critical} part of $p$ given by contribution of residues.  Friesecke and Pego compute $p_*(\xi) = \frac{12(1 + \eps^2 \alpha_1(\eps^2))}{\xi^2 + \eps^2(1 + \eps^2 \beta_1(\eps^2))}$ where $\alpha_1$ and $\beta_1$ are analytic functions.  Lemma 3.2 in \cite{PF1} establishes that 
$$| p(\xi) - p_*(\xi)| < \frac{C_*}{1 + \xi^2} \qquad \mbox{ for } | \mathrm{Im} \xi | < b_* \qquad \mbox{ with } b_* < \kappa$$
Here the symbol $p$ has poles at $\pm i\kappa$ and is analytic on the strip $| \mathrm{Im} \xi | < \kappa$. 

One may now define $p_0(\xi) := \frac{12}{\xi^2 + \eps^2}$ as in the statement of Lemma 3.3 in \cite{PF1} and compute 
$$ p_*(\xi) - p_0(\xi) = \frac{12 \alpha_1(\eps^2)}{(\xi/\eps)^2 + 1 + \eps^2 \beta_1(\eps^2)} - \frac{\beta_1(\eps^2)}{((\xi/\eps)^2 + 1)((\xi/\eps)^2 + 1 + \eps^2\beta_1(\eps^2))}.$$ 
In particular $|p_*(\eps \xi) - p_0(\eps \xi)|$ is bounded uniformly in a strip of width $\eta$ for any $\eta < 1$.  In particular, we take $\eta = 1 / 2$.

It now follows from the triangle inequality that $|p(\eps \xi) - p_0(\eps \xi)|$ is bounded uniformly on $| \mathrm{Im} \xi | < 1/2$.  Now observe that $p^{\eps}(\xi) - p^0(\xi) = \eps^2 (p(\eps \xi) - p_0(\eps \xi))$.  This proves the first assertion in the statement of the lemma.

To prove the second assertion differentiate $p^0(\xi) = \frac{12}{\xi^2 + \beta}$ and $p^\eps(\xi) = \frac{12 \sinc^2(\eps \xi/2)} {\beta + \frac{12}{\eps^2}( 1 - \sinc^2(\eps \xi / 2))}$ with respect to $\beta$ to obtain 
$$\partial_\beta^k p^0(\xi) = (-1)^k k! (\xi^2 + \beta)^{-(k - 1)} p^0(\xi) \qquad \mbox{and} \qquad \partial_\beta^k p^\eps(\xi) = (-1)^k k! (\beta + \frac{12}{\eps^2}(1 - \sinc^2(\eps \xi/2)))^{-(k-1)}p^\eps(\xi).$$  
In light of this computation, the fact that $\frac{12}{\eps^2}(1 - \sinc^2(\eps \xi /2)) = \xi^2 + \mathcal{O}(\eps^2)$, and the fact that we have proven the first assertion in the statement of the lemma, the second assertion now follows.  

To prove the third assertion in the statement of the lemma, write $N^{(\eps)}(x) = N^{(0)}(x) + \frac{1}{2}x^2 \eta(\eps^2 x)$ and compute
$\partial(N^{\eps}(x) - N^0(x)) = x\eta(\eps^2 x) + \frac{1}{2} \eps^2 x^2 \eta'(\eps^2 x)$ and 
$\partial^2(N^\eps(x) - N^0(x)) = \eta(\eps^2 x) + 2\eps^2 x \eta'(\eps^2 x) + \frac{1}{2} \eps^4 x^2 \eta''(\eps^2 x)$.  Also compute $N^0(x) - N^0(y) = \frac{1}{2}(x + y)(x - y)$, $\partial N^0(x) - \partial N^0(y) = x - y$ and $\partial^2 N^0(x) - \partial^2 N^0(y) = 0$.  Now use the triangle inequality together with the fact that $\eta$ is smooth with $\eta(0) = 0$ and the fact that to bound the $H^s$ norm of a multiplication operator it suffices to bound the multiplier to conclude that 
$$\| \partial^k N^\eps (x) - \partial ^k N^0(y) \|_{H^1_a} \le \| x - y\|_{H^1_a} + K\eps^2 \qquad \mbox{ uniformly for }  \| x \|_{H^s_a} + \| y \|_{H^s_a} < 2 \qquad \mbox{ with } k = 0,1,2  $$
as desired.

To prove the fourth assertion in the statement of the lemma, let $P_a$ denote the pseudodifferential operator with symbol $p_a(\xi) = p^0(\xi + ia) = \frac{12}{(\xi+ia)^2 + \beta}$.  We claim that 
\begin{equation}
e^{a \cdot} (I - P^0 \partial N^0(\phi_\beta))^{-1} e^{-a \cdot} = (I - P_a \partial N^0(\phi_\beta))^{-1}.
\label{eq:normeq}
\end{equation}
To prove the claim, first observe that for any operator $L$ with $1 \in \rho(L)$ we have $e^{a \cdot} (I - L)^{-1} e^{-a \cdot} = (I - L_a)^{-1}$ where $L_a = e^{a \cdot} L e^{-a \cdot}$.  Thus it suffices to show that $e^{a \cdot} P^0 \partial N(\phi_\beta) e^{-a \cdot} = e^{a \cdot} P^0 e^{-a \cdot} \partial N(\phi_\beta)$, i.e. that the operator given by multiplication with $e^{-ax}$ commutes with $\partial N^0(\phi_\beta)$.  But $\partial N^0(\phi_\beta)$ is just multiplication by $\phi_\beta$ and multiplication operators commute, so the claim is established.

In light of $\eqref{eq:H1a}$ is suffices to show that $Q_a := (I - P_a \partial N^0(\phi_\beta))^{-1}$ is bounded uniformly in $\mathcal{L}(E^s)$.  Write
$$Q_a  = \left(I - (I-P^0\partial N^0(\phi_\beta))^{-1}(P_a - P^0)\partial N(\phi_\beta)\right)^{-1}(I - P^0 \partial N^0 (\phi_\beta))^{-1} $$ and use the fact that the symbol $p^0$ is analytic in a strip $| \mathrm{Im} z |  < b$, to see (for example from the Neumann series) that the function $a \mapsto \| Q_a \|_{\mathcal{L}(H^1)}$ is continuous on an interval of the form $[0,a_{max}]$, hence achieves its maximum.  Thus $\|Q_a\|$ is bounded uniformly for $a < a_{max}$, as desired.  This proves the fourth assertion in the lemma and thus completes the proof of the lemma.
\end{proof}

\subsection{Uniform Bounds in the Small $\eps$ Regime for the Constant in the Toda Semigroup Decay Estimate}
To prove Proposition $\ref{prop:Toda_sg}$ we rely heavily on \cite{MP}, which itself makes use of \cite{PF4}.  The main result of \cite{MP} is $\eqref{eq:decay_toda}$, in fact all of the claims of Proposition $\ref{prop:Toda_sg}$ are proven in \cite{MP} except the claims that $a = \mathcal{O}(\eps)$ and $b = \mathcal{O}(\eps^3)$ which are proven in \cite{PF1} and \cite{PF4} respectively and the claim that $K$ may be chosen uniformly in $c$, which we prove in this paper.  A refinement of the proof used in \cite{MP} yields the result.  The method of proof is as follows.  Lemma 3 of \cite{MP} shows that the linear evolution equation
\begin{equation}
w_t = JH''(0)w \label{eq:Toda_0}
\end{equation}
is exponentially stable in weighted spaces, i.e. admits semigroup estimates of the form $\eqref{eq:decay_toda}$.  Proposition 7 in \cite{MP} constructs a B\"{a}cklund transformation which conjugates the flow of $\eqref{eq:Toda_lin}$ on the space of exponentially localized profiles which satisfy $\eqref{eq:toda_perp}$ with that of $\eqref{eq:Toda_0}$ on the uniform space $\ell^2 \times \ell^2$.  The important observation that it suffices to study the B\"{a}cklund transformation at $t = 0$ is made in Corollary 8.  Finally, in Corollary 10 of \cite{MP} it is shown that the B\"{a}cklund transformation and its inverse are bounded, thus the estimates $\eqref{eq:decay_toda}$ which are valid for $\eqref{eq:Toda_0}$ are necessarily also valid (with some constant $K(c)$ which may now depend on $c$) for the equation $\eqref{eq:Toda_lin}$.  With this in mind, to prove Proposition $\ref{prop:Toda_sg}$ it suffices to show that the B\"{a}cklund transformation and its inverse are bounded uniformly for $c > 1$.  

Let $\kappa = \kappa(c)$ be the unique root of $\sinh \kappa = \kappa  c$, let $Q_c(n,t) = \log \frac{\cosh(\kappa (n-ct))}{\cosh(\kappa(n - ct +1))}$ be the Toda soliton with speed $c$ and let $A : \ell^2_a \to \ell^2_a$ be the multiplication operator $e^{Q_c(\cdot,0)}$, given more explicitly by $[Ax]_n = \frac{\cosh(\kappa n)}{\cosh(\kappa(n+1))} x_n$.  Because we work in slightly different coordinates than \cite{MP}, the B\"{a}cklund transformation that we study is $\Psi_c$ given in the notation of \cite{MP} by $\Psi_c = \Lambda \Phi_c(0) \Lambda^{-1}$ and more explicitly by $\Psi_c(r,p) = (r',p')$ and $\Psi_c^{-1}(r',p') = (r,p)$ where 
\begin{equation}
\left\{ \ba{l} 
C\delta^{-1} r' = p + \bar{C} \delta^{-1} r \\ \\
p' = \tilde{C}\delta^{-1} r - \hat{C}\delta^{-1} r 
\ea \right.
\qquad \mbox{ and } \qquad
\left\{ \ba{l}
\hat{C} \delta^{-1} r = \tilde{C}\delta^{-1} r' - p' \\ \\
p = C\delta^{-1} r' - \bar{C} \delta^{-1} r
\ea \right.
\label{eq:psi_def}
\end{equation}
Here 
\begin{equation}
\delta := (S - I), \qquad \hat{C} := A - SA^{-1}, \qquad C := A - A^{-1}S^{-1}, \qquad \tilde{C} := A - SA^{-1}S^{-1}, \mbox{ and } \bar{C} = A - A^{-1}.
\label{eq:Cdef}
\end{equation}
The coordinates $(r,p)$ evolve according to the Toda model linearized about a solitary wave with speed $c$ and must satisfy the orthogonality condition $\eqref{eq:toda_perp}$ while the coordinates $(r',p')$ evolve according to the Toda model linearized about the $0$ solution and may live anywhere in $\ell^2_a(\Z^2)$.

Equation $\eqref{eq:psi_def}$ is exactly equation $(29)$ in \cite{MP}, rewritten in $(r,p)$ coordinates.  Note that $\Psi_c$ does depend on $c$.  Other than $\delta$, each of the operators defined in $\eqref{eq:Cdef}$ depends on $A$ which is defined in terms of $\kappa$, which in turn is defined in terms of $c$.

Proposition $\ref{prop:Toda_sg}$ may now be regarded as a corollary of the following lemma, together with the fact that $\|\delta^{-1}\|_{\mathcal{L}(\ell^2_a,\ell^2_a)} = \frac{1}{1-e^{-a}}$, that $\kappa = \mathcal{O}(\eps)$ and that $a = \mathcal{O}(\eps)$.
\begin{lemma} \label{lem:psi_inv}
Let $\delta$, $C$ $\tilde{C}$, $\hat{C}$, and $\bar{C}$ be defined by $\eqref{eq:Cdef}$.  Then there is a constant $K$, independent of $c$ such that the following hold
\begin{enumerate}
\item  $\| \bar{C} \|_{\mathcal{L}(\ell^2_a)} + \| \tilde{C} \|_{\mathcal{L}(\ell^2_a)} < K \kappa$.  
\item $\| C\delta^{-1} \|_{\mathcal{L}(\ell^2_a)} + \| \hat{C}\delta^{-1} \|_{\mathcal{L}(\ell^2_a)} < K$.  
\end{enumerate}
Suppose in addition that $r$, $p$, $r'$, and $p'$ satisfy equations $\eqref{eq:psi_def}$ with $(r,p)$ satisfying $\eqref{eq:toda_perp}$.  Then 
\begin{enumerate}
\setcounter{enumi}{2}
\item $\| r' \|_{\ell^2_a} < K \| p + \bar{C}\delta^{-1} r \|_{\ell^2_a}$.  
\item $\| r \|_{\ell^2_a} < K \|\tilde{C}\delta^{-1}r' - p'\|_{\ell^2_a}$.
\end{enumerate}
\end{lemma}

\begin{proof}[Proof of Lemma $\ref{lem:psi_inv}$]
We bound each expression in turn.  Let $q_n = \cosh(n \kappa)$.  Then $[\bar{C} x]_n = \frac{q_n^2 - q_{n+1}^2}{q_nq_{n+1}} \approx \pm 2 \sinh \kappa \approx \pm 2 \kappa$.  Similarly
$[\tilde{C}x]_n = \frac{q_n - q_{n+2}}{q_{n+1}} \approx \pm 2 \sinh \kappa \approx \pm 2 \kappa$.
We write $C\delta^{-1} = \bar{C}\delta^{-1} + A^{-1}S^{-1}$ to bound $C\delta^{-1}$.  Also write
$\hat{C}\delta^{-1} = \bar{C} + A^{-1}S^{-1} + (A^{-1}S - SA^{-1})\delta^{-1}$.  The first two of these terms have already been bounded.  To bound the last term note
$[(A^{-1}S - SA^{-1})x]_n = \frac{q_{n+1}^2 - q_nq_{n+2}}{q_n q_{n+1}} = o(\kappa)$.  This completes the proof of $1.$ and $2.$ above.

To establish $3$ and $4$ we solve the equations $C\delta^{-1}x = y$ and $\hat{C}\delta^{-1}x = y$ for $x$ in terms of $y$.  This is not always possible because $C$ has a one-dimensional kernel and its formal adjoint $\hat{C}$ has a one-dimensional cokernel.  This will manifest in our expressions for $x$ as either a consistency condition that $x$ must satisfy or the presence of a free parameter in the solution for $x$.  However, it is known that the map $\Psi_c$ defined by $\eqref{eq:psi_def}$ is an isomorphism from the subspace of $\ell^2_a(\Z^2)$ on which $\eqref{eq:toda_perp}$ holds to $\ell^2_a(\Z^2)$.  Thus in the context of conditions 3. and 4., the consistency equation will always hold and the free parameter will be fixed by $\eqref{eq:toda_perp}$. 

We now invert $C \delta^{-1} x = y$.  The equation $Cz = y$ is nothing more than a non-autonomous first order linear difference equation and thus may be solved (e.g. with a discrete integrating factor) to obtain 
$q_{M+1}^2 z_M - q_{m+1}^2 z_m = \sum_{n = m}^{M-1} \frac{q_{n+1}^3}{q_n} y_n$ which is valid for any integers $m$, $M$.  Taking either $M \to \infty$ OR $m \to -\infty$ we obtain the expression
$z_k = -\sum_{n = k}^\infty \frac{q_{n+1}^3}{q_n q_k^2} y_n = \sum_{n = -\infty}^{k-1} \frac{q_{n+1}^3}{q_n q_k^2} y_n$.  The fact that the range of $C$ has codimension $1$ manifests in the condition that $Cz = y$ may only be solved when $\sum_{n \in \Z} \frac{q_{n+1}^3}{q_n} y_n = 0$.  One may show that this condition is satisfied for $y = p + \bar{C}\delta^{-1}r$ so long as $(r,p)$ satisfy $\eqref{eq:toda_perp}$.  Applying $\delta$ to the result we see that 
\begin{equation}
x_k = \frac{q_{k+1}^3}{q_k^3} y_k + \frac{q_{k+1}^2 - q_k^2}{q_k^2q_{k+1}^2} \sum_{n = k+1}^\infty \frac{q_{n+1}^3}{q_n} y_n  = \frac{q_{k+1}}{q_k} y_k + \frac{q_k^2 - q_{k+1}^2}{q_k^2q_{k+1}^2} \sum_{n = -\infty}^{k-1} \frac{q_{n+1}^3}{q_n} y_n. 
\label{eq:Cinv}
\end{equation}
One may now choose the first representation of $x$ for $k \ge 0$, the second for $k < 0$, regard each as a convolution by changing the index $n$ to $k-n$, estimate that the coefficients in the sums in $\eqref{eq:Cinv}$ are bounded by $C\kappa e^{-2 \kappa}$, and finally apply the Hausdorff-Young inequality for convolutions to see that
$$ \| x \|_{\ell^2_a} \le C(1 + \kappa \| e^{-(2\kappa -a)\cdot} \|_{\ell^1}) \| y \|_{\ell^2_a} \le C \| y\|_{\ell^2_a}.$$
In the last inequality we have used that $a < 2\kappa$.  This establishes assertion 3.  The proof of $4.$ is similar.  We may solve $\hat{C} z = y$ to obtain
$$z_k = \left\{ \ba{ll} \frac{q_{k+1}q_{k+2}}{q_1 q_2} z_0 + \sum_{n = 0}^{k-1} \frac{q_{k+1}q_{k+2}}{q_{n+2}^2} y_n & k > 0 \\ \\ 
\frac{q_{k+1}q_{k+2}}{q_1q_2} z_0 - \sum_{n = k}^{-1} \frac{q_{k+1}q_{k+2}}{q_{n+2}^2} y_n & k < 0
\ea \right. 
$$
Here $z_0$ appears as a free parameter which is fixed by the orthogonality condition $\eqref{eq:toda_perp}$ in our case, i.e. when $z = \delta^{-1}r$ and $\sum_{n \in \Z} \frac{q_{n+1}^3}{q_n} y_n = 0$ for $y = p + \bar{C}\delta^{-1}r$.  One may now prove assertion 4 using ideas similar to the above.

\end{proof}

\subsection{Uniform Bounds in the Small $\eps$ Regime for the Neighborhood on which Tubular Coordinates are Valid}

\begin{proof}[Proof of Proposition $\ref{Ptube}$]

Let $T_0$ be given as in Lemma $\ref{L6}$.  Assume the conditions 
$$ (i) \; |c_\alpha - c_\alpha^*| < C\eps^{2 + \eta} \qquad  (ii) \;  \alpha \tau_\alpha > T_0 $$
and define
$$\mathcal{Q} := \{ (\tau_+,\tau_-,c_+,c_-,u,\tilde{u}) \in \R^4 \times \ell^2(\Z,\R^2) \times \ell_a^2(\Z,\R^2) \cap \ell^2_{-a}(\Z,\R^2) \; | \; (i) \mbox{ and } (ii) \mbox{ hold} \}.$$
Define $F : \mathcal{Q} \to \R^4$ by $F(\tau,c,u,\tilde{u}) = \mathrm{col} \left(F_{1+},F_{2+},F_{1-},F_{2-} \right)$ with $ F_{1,\alpha} = \omega_\alpha(\xi_{1,\alpha},(u-\hat{u})h_\alpha)$ and $F_{2,\alpha} = \omega_\alpha(\xi_{2,\alpha}(\tilde{u}-\hat{u})h_\alpha)$.
Here $\xi_{i,\alpha}$ is evaluated at $(\tau_\alpha,c_\alpha)$ and $\hat{u}$ is evaluated at $(\tau_+,\tau_-,c_+,c_-)$, which we abbreviate as $(\tau,c)$.  Observe that if we regard $u$ as the solution of $\eqref{FPU}$ under study and $v_1$ as another exact solution of $\eqref{FPU}$ intended to capture the part of the perturbation which is not exponentially localized, then upon denoting 
$$v_2 := \tilde{u} - \hat{u},  \qquad v_1 := u - \tilde{u}, \qquad \mbox{and} \qquad v_{i,\alpha} = v_i h_\alpha \mbox{ for } i=1,2$$ 
then $v_{2,\alpha}$ satisfies the orthogonality conditions $\eqref{Perp1}$ and $\eqref{Perp2}$ for $\alpha \in \{+,-\}$ if and only if $F(\tau,c,u,\tilde{u}) = 0$.  We remark that it is natural to proceed by applying the implicit function theorem to find the zero set of $F$ near $(\tau^*,c^*,\hat{u},\hat{u})$.  In fact, this has been done in \cite{Mizumachi} and \cite{PF2}.  However, small $\eps$ asymptotics are not of interest there, and the naive estimates from the implicit function theorem on the size of the neighborhood on which these coordinates are valid are not sufficient for our purposes.  Rather, we construct an explicit contraction mapping and exploit the structure of $F$ to show that the coordinates $\tau$, $c$, and $v_2$ are defined so long as $\eqref{eq:usmall}$ holds.

Before we proceed it is convenient to establish the small $\eps$ asymptotics for $F$ and its derivatives.  Let $A$ be the matrix studied in Lemma $\ref{L6}$ and let \\ \\
$C_\alpha = \left( \ba{cc} \omega_\alpha( \partial_\tau \xi_{1,\alpha}, (u - \hat{u})h_\alpha) & 
- \omega_\alpha( \partial_c \xi_{1,\alpha}, (u - \hat{u})h_\alpha) \\ 
\omega_\alpha(\partial_\tau \xi_{2,\alpha}, (\tilde{u} - \hat{u})h_\alpha) & \omega_\alpha(\partial_c \xi_{2,\alpha}, (\tilde{u} - \hat{u})h_\alpha) \ea \right)$ and $C = \left( \ba{cc} C_+ & 0 \\ 0 & C_- \ea \right)$ so that \\
$D_{(\tau,c)} F (\tau,c,u,\tilde{u}) = A + C + {\tt exp}$.
Use $\eqref{eq:power_rules}$ to see that when 
\begin{equation} 
\| u - \hat{u} \| + \sum_{\alpha \in \{+,-\}} \| (\tilde{u} - \hat{u})h_\alpha \|_\alpha \le K \eps^{5/2 + \eta} \qquad \eta > 0,
\label{eq:usmall2}
\end{equation}
which corresponds to $\eqref{eq:usmall}$, the terms in $C$ are dominated by the terms in $A$ (except for the upper left term, whose counterpart in $A$ is exponentially small).  Thus so long as $\eqref{eq:usmall2}$ holds, it follows that for $\eps$ sufficiently small $D_{(\tau,c)}F$ is invertible.  Moreover, to leading order in $\eps$, the off-diagonal blocks of $(D_{(\tau,c)}F)^{-1}$ are exponentially small and the diagonal blocks have the structure
$\left( \ba{ll} \eps^{-4} & \eps^{-1} \\ \eps^{-1} & \eps^{2+\eta} \ea \right)$ which is similar to $\eqref{eq:Ainv_small_eps}$.
Similarly, one can compute 
$
D^2 F_{1,\alpha} = \left( \ba{ll} 
\partial_\tau^2 F_{1,\alpha} & \partial^2_{\tau,c} F_{1,\alpha} \\
\partial^2_{\tau,c}F_{1,\alpha} & \partial^2_c F_{1,\alpha} \ea \right)$ and
$D^2 F_{2,\alpha} = \left( \ba{ll} \partial_\tau^2 F_{2,\alpha} & \partial^2_{\tau,c} F_{2,\alpha} \\ \partial^2_{\tau,c}F_{2,\alpha} & \partial^2_c F_{2,\alpha} \ea \right)$
where
$\partial_\tau^2 F_{1,\alpha} = 
\omega_\alpha( \partial_\tau^2 \xi_{1,\alpha}, (u - \hat{u})h_\alpha) + \omega_\alpha(\partial_\tau \xi_{1,\alpha},\xi_{1,\alpha})$, 
$\partial^2_{\tau,c} F_{1,\alpha} = 
\omega_\alpha(\partial^2_{(\tau,c)} \xi_{1,\alpha}, (u - \hat{u})h_\alpha) -\omega_\alpha(\partial_\tau \xi_{1,\alpha}, \xi_{2,\alpha})$,
$\partial^2_{c} F_{1,\alpha} =  
\omega_\alpha(\partial_c^2 \xi_{1,\alpha}, (u - \hat{u})h_\alpha) + \omega_\alpha(\partial_c \xi_{1,\alpha}, \xi_{2,\alpha}) + \omega_\alpha(\xi_{1,\alpha},\partial_c^2 \xi_{2,\alpha})
$
and
$\partial_\tau^2 F_{2,\alpha} = \omega_\alpha(\partial_\tau^2 \xi_{2,\alpha}, (\tilde{u} - \hat{u})h_\alpha) - 
\omega_\alpha( \partial_\tau \xi_{2,\alpha},\xi_{1,\alpha}h_\alpha) 
- \omega_\alpha(\xi_{2,\alpha}, \partial_\tau \xi_{1,\alpha})$,
$\partial^2_{(\tau,c)} F_{2,\alpha} =  
\omega_\alpha( \partial^2_{(\tau,c)} \xi_{2,\alpha} , (\tilde{u} - \hat{u})h_\alpha) - \omega_\alpha(\partial_\tau \xi_{2,\alpha},\xi_{2,\alpha}h_\alpha) - 
\omega_\alpha(\partial_c \xi_{2,\alpha}, \xi_{1,\alpha}h_\alpha) -
\omega_\alpha(\xi_{2,\alpha},\partial_c \xi_{1,\alpha})$
and
$\partial_c^2 F_{2,\alpha} = \omega_\alpha(\partial_c^2 \xi_{2,\alpha}, (\tilde{u} - \hat{u})h_\alpha) + \omega_\alpha(\partial_c \xi_{2,\alpha}, \xi_{2,\alpha})$.

Use $\eqref{eq:power_rules}$, $\eqref{eq:usmall2}$, and $\eqref{eq:OmegaBdd}$ to obtain the leading order expressions
\begin{equation}
D^2 F_{1,\alpha}  \sim \left( \ba{ll} \eps^5 & \eps^2 \\ \eps^2 & \eps^{-1} \ea \right) \qquad 
D^2 F_{2,\alpha} \sim \left( \ba{ll} \eps^2 & \eps^{-1} \\ \eps^{-1} & \eps^{-4} \ea \right).
\label{eq:2nd_derivative_f}
\end{equation}

Note that to control the terms involving $\omega_\alpha(\cdot, (u-\hat{u})h_\alpha)$ we have used the fact that the first term in the symplectic product contains a factor of $\xi_{1,\alpha} = \frac{1}{c_\alpha} JH'(u_{c_\alpha})$.  The presence of $J$ allows us to replace the symplectic inner product with a Euclidean one and then use the regular Euclidean Cauchy-Schwartz inequality rather than $\eqref{eq:OmegaBdd}$.  This is important because $u - \hat{u}$ need not lie in the weighted space.

As in the proof of the implicit function theorem we set up a contraction mapping.  It will be convenient to let $A := D_{(\tau,c)} F(\tau^*,c^*,u,\tilde{u})$, to let $x$ denote the variable $(\tau,c)$, and to fix $u$ and $\tilde{u}$ and to denote $F(x,u,\tilde{u})$ simply as $F(x)$.  Define $G(x) := A^{-1}(Ax - F(x))$ so that fixed points of $G$ correspond to zeros of $F$.  Write
\[ \ba{lll}
(F(x)-F(y) - A(x-y)) & = & \left[\int_0^1 \left(F'(tx + (1-t)y) - F'(x^*)\right)dt\right](x-y) \\ 
& = & \langle \int_0^1 \left( \int_0^1  F''(stx + s(1-t)y) ds \right) (tx + (1-t)y - x^*) dt ,(x-y) \rangle. \ea
\]
We renorm $\R^4$ by $|(\tau_+,\tau_-,c_+,c_-)|_r := \sqrt{\eps^3(\tau_+^2 + \tau_-^2) + \eps^{-3}(c_+^2 + c_-^2)}$.

Restrict attention to the neighborhood of $x^*$ on which $|(\tau,c) - (\tau^*,c^*)|_r < \eps^{1 + 2\eta}$ i.e. on which
\begin{equation}
|\tau - \tau^*| < \eps^{-1 + \eta} \qquad \mbox{ and } \qquad |c - c^*| < \eps^{2 + \eta}
\label{eq:mod_small}
\end{equation}
hold.  Denote the $\tau$ and $c$ components of $x - y$ by $\theta_\tau$ and $\theta_c$ respectively.
Then using $\eqref{eq:2nd_derivative_f}$ and $\eqref{eq:mod_small}$ we have
\[ G(x) - G(y) \sim
\left( \ba{cc} \eps^{-4} & \eps^{-1} \\ \eps^{-1} & \eps^{2 + \eta} \ea \right)
\left(\ba{l} \eps^{4 + \eta} \theta_\tau + \eps^{1 + \eta} \theta_c \\
\eps^{1 + \eta} \theta_\tau + \eps^{-2 + \eta} \theta_c \ea \right) 
 \sim  \left( \ba{c} \eps^\eta \theta_\tau + \eps^{3 + \eta} \theta_c \\ 
\eps^{3 + \eta} \theta_\tau + \eps^\eta \theta_c \ea \right).
\] Note that we have absorbed the off-diagonal terms which are exponentially small.  Thus
\begin{equation} | G(x) - G(y) |_r^2 \sim \eps^{3 + 2 \eta} \theta_\tau^2 + \eps^{-3 + 2 \eta} \theta_c^2 = \eps^{2 \eta} |(\theta_\tau,\theta_c)|_r^2. \label{eq:Glip}
\end{equation}
Moreover $G(0) \sim \left( \ba{cc} \eps^{-4} & \eps^{-1} \\ \eps^{-1} & \eps^{2 + \eta} \ea \right) \left( \ba{c} \eps^{3/2} \| u - \hat{u} \| \\ \eps^{-3/2} \| (\tilde{u} - \hat{u})h_\alpha \|_\alpha \ea \right) \sim \left( \ba{c} \eps^{\eta} \\ \eps^{3 + \eta} \ea \right)$ where in the last relation we have used $\eqref{eq:usmall2}$.  Thus $|G(0)|_r \sim \eps^{3/2 + \eta}$.
It now follows from the triangle inequality and $\eqref{eq:Glip}$ that for $\eps$ small enough $G$ maps a neighborhood in $\R^4$ of the form $\eqref{eq:mod_small}$ to itself and is moreover a contraction on this neighborhood.  In particular $G$ has a unique fixed point $x = (\tau,c)$ in this neighborhood.  Moreover, as is standard in implicit function theorems, we can let $u$ and $\tilde{u}$ vary in the defintion of $G$ and the fixed point $(\tau,c) = (\tilde{\tau}(u,\tilde{u}),\tilde{c}(u,\tilde{u}))$ is smooth in $u$ and $\tilde{u}$ by the uniform contraction principle.

It therefore follows that these functions $\tau_\alpha = \tilde{\tau}_\alpha(u,\tilde{u})$ and $c_\alpha = \tilde{c}_\alpha(u,\tilde{u})$ map a neighborhood of $(\hat{u}(\tau^*,c^*),\hat{u}(\tau^*,c^*))$ to a neighborhood of $(\tau^*,c^*)$ such that $F(\tau,c,u,\tilde{u}) = 0$ if and only if $c_\alpha = \tilde{c}_\alpha(u,\tilde{u})$ and $\tau_\alpha = \tilde{\tau}_\alpha(u,\tilde{u})$.  Moreover, the implicitly defined functions are unique and as smooth as $F$, i.e. $C^\infty$.
Thus upon defining $\tilde{v_2}(u,\tilde{u}) := \tilde{u} - \hat{u}(\tilde{\tau}(u,\tilde{u}),\tilde{c}(u,\tilde{c}))$ we obtain a homeomorphism between a neighborhood $\mathcal{U}$ of $(\hat{u}(\tau^*,c^*),\hat{u}(\tau^*,c^*))$ to a neighborhood $\mathcal{V}$ of $(\tau^*,c^*,0)$.

Recall that $\tau^*$ and $c^*$ are arbitrary, thus for each $\tau$ and $c$  satisfying $(i)$ and $(ii)$ we have a homeomorphism between $\mathcal{U}(\tau,c)$ and $\mathcal{V}(\tau,c)$.  We now show that these local homeomorphisms may be patched together to produce a global homeomorphism betwenn $\mathcal{U}^*$ and $\mathcal{V}^*$.

To that end, suppose that 
$$\hat{u}(\tau,c) + v_2 = \hat{u}(\bar{\tau},\bar{c}) + \bar{v}_2$$
so that 
$$ \hat{u}(\tau,c) - \hat{u}(\bar{\tau},c) + \hat{u}(\bar{\tau},c) - \hat{u}(\bar{\tau},\bar{c}) = \bar{v}_2 - v_2$$
by the mean value theorem there are numbers $\tau_\alpha^\dagger$ between $\tau_\alpha$ and $\bar{\tau}_\alpha$ and $c_\alpha^\dagger$ between $c_\alpha$ and $\bar{c}_\alpha$ such that 
$$ \xi_1(\tau_+^\dagger,c_+)(\tau_+ - \bar{\tau}_+) + \xi_1(\tau_-^\dagger,c_-)(\tau_- - \bar{\tau}_-) + \xi_2(\bar{\tau}_+,c_+^\dagger)(c_+ -\bar{c}_+) + 
\xi_2(\bar{\tau}_-,c_-^\dagger)(c_- - \bar{c}_-) = v_2 - \bar{v}_2$$
In light of $\eqref{eq:power_rules}$ we see that so long as we force $\| v_2 h_\alpha\|_\alpha + \|\bar{v}_2 h_\alpha\|_\alpha < Ce^{3/2 + 2\eta}$ it follows that $|\tau - \bar{\tau}| < C\eps^{-1 + 2\eta}$ and $|c - \bar{c}| < C\eps^{2 + 2 \eta}$.  In particular, $|\tau - \bar{\tau}|$ and $|c - \bar{c}|$ are sufficiently small for the local inverses to hold.
\end{proof}

\bibliography{FPU_reference}

\end{document}